\documentclass{amsart}
\usepackage{amsmath,amssymb, amsthm,amsfonts}
\usepackage{graphicx} 
\usepackage{subfiles}
\usepackage{enumitem}
\usepackage[english]{babel}
\usepackage{tabularx}
\usepackage{enumerate}
\newtheorem{definition}{Definition}[section]
\newtheorem{proposition}[definition]{Proposition}
\newtheorem{propositionand}[definition]{Proposition and Definition}
\newtheorem{theoremand}[definition]{Theorem and Definition}
\newtheorem{lemma}[definition]{Lemma}
\newtheorem{theorem}[definition]{Theorem}
\newtheorem{corollary}[definition]{Corollary}

\newtheorem{example}[definition]{Example}
\title{Discrete channel surfaces}
\author{Udo Hertrich-Jeromin, Wayne Rossman and Gudrun Szewieczek}
\begin{document}
\maketitle
\begin{center}
\begin{minipage}{11cm}\small
\textbf{Abstract.} We present a definition of discrete channel surfaces in Lie 
sphere geometry, which reflects several properties for smooth 
channel surfaces.  Various sets of data, defined at vertices,
on edges or on faces, are associated with a discrete channel
surface that may be used to reconstruct the underlying particular
discrete Legendre map.
As an application we investigate isothermic discrete channel 
surfaces and prove a discrete version of Vessiot's Theorem.
\end{minipage}
\vspace*{0.5cm}\\\begin{minipage}{11cm}\small
\end{minipage}
\vspace*{0.5cm}\\\begin{minipage}{11cm}\small
\end{minipage}
\end{center}
%
%
\  \\\section{Introduction}
\noindent From a perspective of higher geometries, channel surfaces form a class of
simple surfaces,
as they can be considered as curves in a suitable space of geometric
objects, e.g.~a space of spheres or a space of Dupin cyclides.  As such, they are well
suited for
applications in engineering or computer aided geometric design, as they
can be
described in a simpler way than more general surfaces.  On the other hand,
they are
of interest from a theoretical point of view, since, for example, they
provide a class
to draw simple but non-trivial examples for complex problems -- in
particular,
when invariance under M\"{o}bius or Lie sphere transformations is sought.

In this paper, we propose a discretization of channel surfaces and, as a
main theorem
and application of our discretization, we prove a discrete version of
Vessiot's
Theorem \cite{vessiot}.

A satisfactory discretization of the notion of a channel surface turns out
to be
surprisingly intricate: due to the various characterizations and geometric
properties
of channel surfaces that we aim to reflect in the discretization, a direct
and
naive approach has proven unable to preserve the desired mathematical
structures.
The natural context of these structures is varied, and we will work
simultaneously
in different geometries: M\"{o}bius geometry, Lie sphere geometry, as well as
Euclidean
geometry.  The various sets of data associated to a channel surface depend
on the
respective setting.  In particular, we find that a geometric object may not
be defined at just one type of cells in the underlying cell complex, but be
distributed across cells of different dimensions.

This approach is based not only on the analysis of enveloped sphere
congruences,
but also on enveloped congruences of Dupin cyclides, cf.~\cite{paper_cyclidic}. 
In particular,
we use a notion of ``Lie cyclide'' for discrete channel surfaces.
Also, the approach we propose carries directly over to a semi-discrete
setting.
\\\textbf{Acknowledgements.} We would like to thank Mason Pember and Joseph Cho for fruitful and enjoyable discussions around this topic. For the main parts of this work the third author was funded by her JSPS Postdoctoral Fellowship P17734. Moreover, this work has been partially supported by the FWF/JSPS Joint Project grant I1671-
N26 ``Transformations and Singularities'' and also by the grants Kiban S (17H06127) and
Kiban C (15K04845) from the Japan Society for the Promotion of
Science.

\subsection{Smooth channel surfaces and problems with the discretization}
Smooth channel surfaces, first introduced by Monge \cite{M1850}, can be characterized  by various equivalent properties \cite{blaschke}: a channel surface is
\vspace*{0.1cm}\begin{itemize}
\item[(S1)] an envelope of a 1-parameter family of spheres,
\item[(S2)] a surface with one family of circular curvature lines,
\item[(S3)] a surface such that one of its principal curvatures is constant along its curvature direction.
\end{itemize}
The equivalence of these properties mainly relies on a fact which follows from  Joachimsthal's Theorem: if a surface is the envelope of a 1-parameter family of spheres, then it envelopes along curvature lines, which are therefore circles.
%
%
\\\\We start by investigating the interplay of the conditions (S1)--(S3) for discrete nets. It will turn out that there are significant differences between the smooth and the discrete theories and, therefore, a naive definition of a discrete channel surface using a discrete version of just one of the smooth properties is unsatisfactory.
%
%
\\\\Let us consider discretizations of surfaces parametrized by curvature line coordinates, namely \textit{principal contact element nets} (cf.~\cite[\S 3.5]{bobenko_book}). These are circular nets equipped with normals that represent ``contact elements'', that is, for each vertex the sphere pencil consisting of spheres with this normal that also contain the vertex.  Moreover, two adjacent normals fulfill the \textit{Legendre condition}, that is, two adjacent contact elements share a common sphere, called the \emph{curvature sphere} for the related edge.  

Recall that a principal contact element net $f:\mathbb{Z}^2 \rightarrow \{ \text{contact elements} \}$ \textit{envelopes} a sphere congruence $s:\mathbb{Z}^2 \rightarrow \{ \text{oriented spheres} \}$ if $s_p \in f_p$ for any $p \in \mathbb{Z}^2$. 
\\\\Hence, from the definition we immediately deduce that the properties (S1) and (S3) are equivalent also in the discrete setting: if $f$ envelopes a 1-parameter family of spheres, they are curvature spheres and therefore the principal curvature along the corresponding curvature line is constant. Conversely, curvature spheres which are constant along one family of curvature lines determine a 1-parameter family of enveloped spheres.
\\\\The following construction yields examples of principal contact element nets, which will further help to illustrate the situation in the discrete setting: let $c_1:\mathbb{Z} \rightarrow \mathbb{R}^3$ be an arbitrary discrete curve equipped with normals $n_1$ fulfilling the Legendre condition. Reflecting the curve and these normals in an arbitrary plane gives a second curve $c_2$ with normals $n_2$ that again satisfy the Legendre condition. Hence, the sequence $((c_1,n_1), (c_2,n_2))$ determines a coordinate ribbon of a principal contact element net. In this way, by repeated reflection in a sequence $(P_i)_{i \in I\subseteq \mathbb{Z}}$ of planes, we obtain a principal contact element net $f:=(c_i, n_i )_{i \in I}$.
\\\\\textbf{Example 1.} Let $c_1$ be a discrete circular curve with normals that interesect at a common point, hence the contact elements along this curve share a common curvature sphere. For any choice of planes $(P_i)_{i \in I}$, we then obtain a principal contact element net, which satisfies the properties (S1)--(S3) (cf.~Figure 1, left).
\\\\\textbf{Example 2.} Let $c_1$ be an arbitrary discrete curve on a sphere with normals $n_1$ going through the center of this sphere and $(P_i)_{i \in I}$ be a sequence of non-parallel planes. Then $f:=(c_i, n_i )_{i \in I}$ yields a discrete principal contact element net, which envelopes a 1-parameter family of spheres, hence satisfies (S1) and (S3). But, if the initial curve $c_1$ is not circular, property (S2) obviously does not hold (cf.~Figure 1, right).
\\\\\begin{minipage}{5cm}
\includegraphics[scale=0.7]{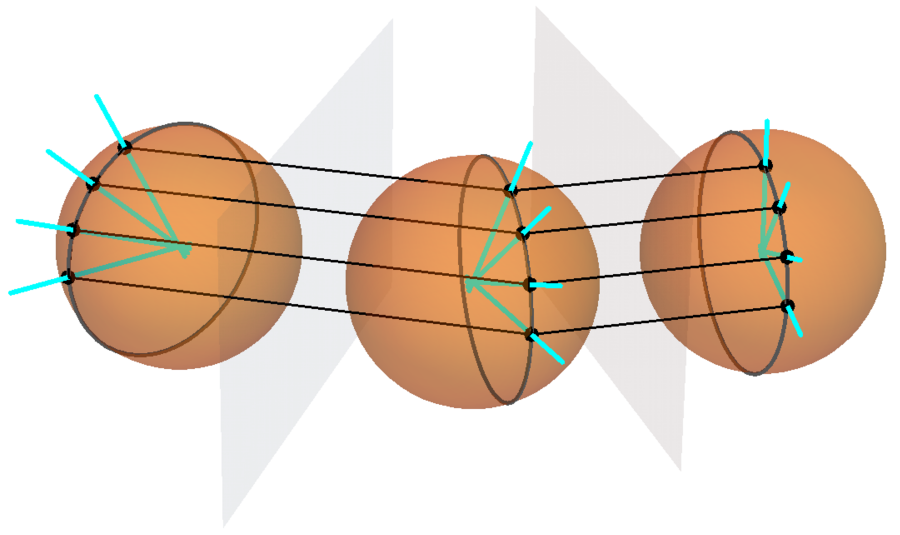}
\end{minipage}
\begin{minipage}{1.9cm}
\ \ \ \ \ \ \ 
\end{minipage}
\begin{minipage}{5cm}
\vspace*{-0.4cm}\hspace*{-0.2cm}\includegraphics[scale=0.74]{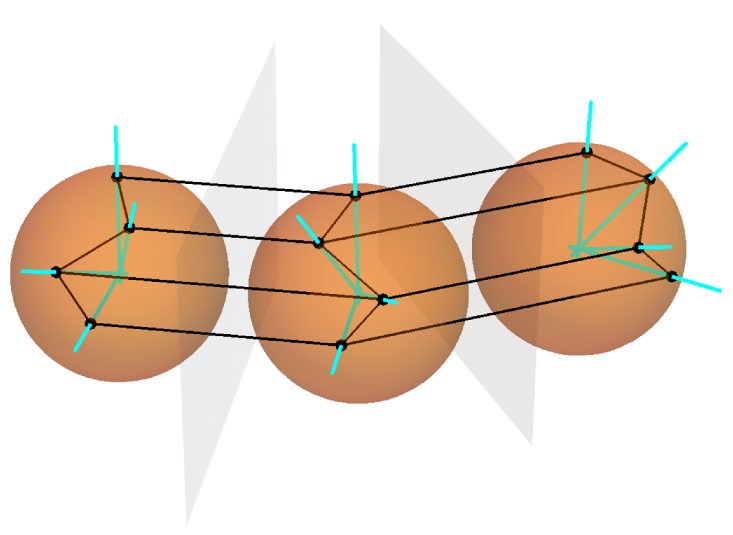}
\end{minipage}
%
%
\\[12pt] \ \\\begin{minipage}{0.5cm}
\ \ 
\end{minipage}
\begin{minipage}{11cm}
\textbf{Fig 1.} \textit{Left:} A principal contact element net, which satisfies properties (S1)-(S3) and therefore qualifies as a ``discrete channel surface''. \textit{Right:} A discrete Legendre map enveloping a 1-parameter family of spheres without circular curvature lines.
\end{minipage}
\begin{minipage}{0.5cm}
\ \ 
\end{minipage}
%
%
%
%
\\\\[6pt]\textbf{Example 3.} Let $c_1$ be a discrete circular curve with normals $n_1$ that do not intersect at the same point and $(P_i)_{i \in I}$ be a sequence of parallel planes. Then the principal contact element net $f:=(c_i, n_i)_{i \in I}$ has a family of circular curvature lines, but does not envelop a 1-parameter family of spheres. Thus, this construction provides examples which contradict the properties (S1) and (S3), while (S2) is fulfilled.
%
%
\\\\In the realm of Lie sphere geometry, Blaschke gave another characterization of smooth channel surfaces using the Lie cyclides of a surface, special Dupin cyclides which make maximal-order contact with the surface along its curvature directions (cf.~\cite{blaschke, trafo_channel}): a surface in Lie sphere geometry is a channel surface if and only if 
\vspace*{0.1cm}  
\begin{itemize}
\item[(S4)] its Lie cyclide congruence is constant along one of its curvature directions.
\end{itemize}
\vspace*{0.1cm}  
In the case of a channel surface the curvature spheres along each circular curvature direction coincide with those of the corresponding constant Lie cyclide. 

This characterization inspires our definition for discrete channel surfaces, which will satisfy discrete versions of the properties (S1)--(S3) for smooth channel surfaces. 
%
%
\subsection{Preliminaries}
\noindent Throughout this paper we will work in the realm of Lie sphere geometry and will use the hexaspherical coordinate model introduced by Lie. In this subsection we briefly summarize the basic principles of this model and fix notations. For more details about the theoretic background, see for example \cite{book_cecil}. 
\\In order to use Lie's model of Lie sphere geometry, we shall consider the vector space $\mathbb{R}^{4,2}$ endowed with a metric $(.,.)$ of signature $(4,2)$. The linear span of vectors will be denoted by $\langle .,...,.\rangle$. 

In this setting the Lie quadric (projective light cone)
\begin{equation*}
\mathcal{L}:=\{ \langle \eta \rangle \subset \mathbb{R}^{4,2} | \ (\eta, \eta)=0  \} \subset \mathbb{P}(\mathbb{R}^{4,2})
\end{equation*}
represents the set of oriented 2-spheres and two 2-spheres are in oriented contact if and only if any corresponding vectors in the light cone are orthogonal. Therefore lines in $\mathcal{L}$ correspond to pencils of 2-spheres which are all in oriented contact to each other, hence parametrise ``contact elements''. This Grassmannian of null 2-planes in $\mathbb{R}^{4,2}$ will be denoted by 
\begin{equation*}
\mathcal{Z}:= \{ \langle \eta, \xi \rangle \subset \mathbb{R}^{4,2} \ | \  (\eta, \xi)=0 \} \subset G_2(\mathbb{R}^{4,2}).
\end{equation*}
The group of Lie sphere transformations preserves oriented contact between spheres, hence, also contact elements. However, parallel transformations are Lie transformations in Euclidean space, hence points (viewed as spheres with radius zero) and spheres are not distinguishable.

We observe how M\"obius geometry arises as a subgeometry in this model. After the choice of a point sphere complex $\mathfrak{p} \in \mathbb{R}^{4,2}$, $(\mathfrak{p},\mathfrak{p})\neq 0$, that is, after we have decided which spheres have radius zero and are therefore ``points'' (called point spheres), we obtain a M\"obius geometry modeled on $\langle \mathfrak{p} \rangle ^\perp$. Namely, those Lie sphere  transformations that fix the point sphere complex $\mathfrak{p}$ are exactly the conformal transformations of $\langle \mathfrak{p} \rangle^\perp$.
\\\\Furthermore, if additionally another vector $\mathfrak{q} \in \mathbb{R}^{4,2} \setminus \{ 0 \}$ is fixed, a Riemannian or Lorentzian space form geometry can also be treated as a subgeometry of Lie sphere geometry. To this end, we consider the space form 
\begin{equation*}
\mathcal{Q}^3:=  \{ \eta \in \mathbb{R}^{4,2} \ | \ (\eta, \eta)=0, \ (\eta, \mathfrak{q})= -1, \ (\eta, \mathfrak{p})=0  \},
\end{equation*} 
a 3-dimensional quadric of constant sectional curvature $-(\mathfrak{q},\mathfrak{q})$.
%
%
%
%
%
\subsection{Discrete Legendre maps}
Discrete surfaces studied in this paper are represented by discrete Legendre maps from a connected quadrilateral cell complex $\mathcal{G}=(\mathcal{V},\mathcal{E},\mathcal{F})$ to the space of contact elements $\mathcal{Z}$. The set of vertices (0-cells), edges (1-cells) and faces (2-cells) of the cell complex $\mathcal{G}$ are denoted by $\mathcal{V}$, $\mathcal{E}$ and $\mathcal{F}$. Moreover, we assume that interior vertices are of even degree, that is, the number of edges $e_i \in \mathcal{E}$ incident to a vertex $v \in \mathcal{V}$ is even. 

Then there exists a consistent labelling of the edges $\mathcal{E}$ of $\mathcal{G}$ such that two opposite edges of a face are labelled with '$+$' and the other two edges of the face are labelled with '$-$'. In this way, we obtain a partition $\mathcal{E}^+  \overset{.}{\cup} \mathcal{E}^-=\mathcal{E}$ of the set of edges. 
\\\\Paths along edges with the same label '$\pm$' will be called \emph{$\pm$-coordinate lines}, while sequences of faces consecutively sharing common  edges with the same edge-label '$\mp$' are said to be \emph{$\pm$-coordinate ribbons}. 
\\\\\begin{minipage}{5cm}
\hspace*{3cm}\includegraphics[scale=0.2]{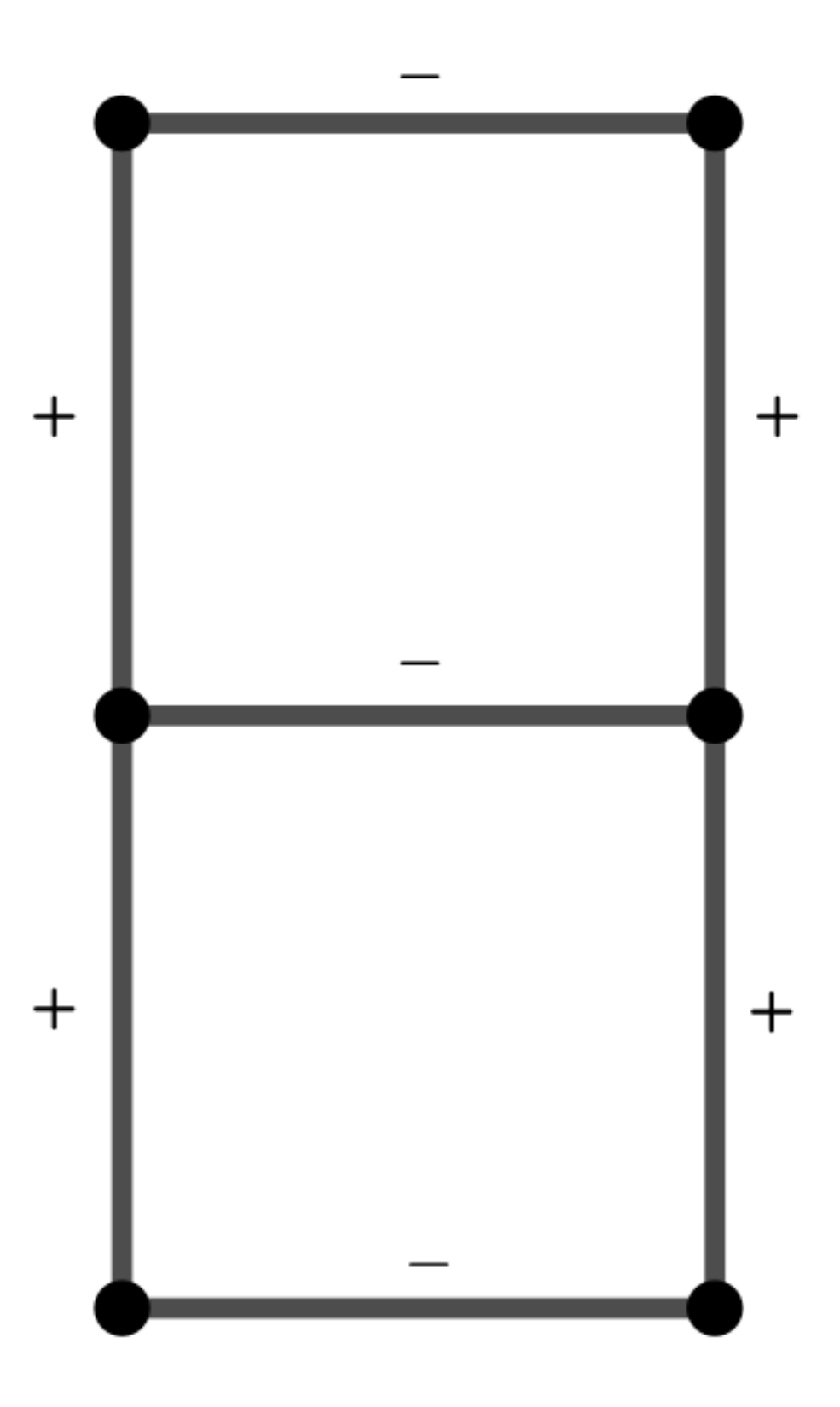}
\end{minipage}
\begin{minipage}{0.8cm}
\ \ 
\end{minipage}
\begin{minipage}{5cm}
\includegraphics[scale=0.4]{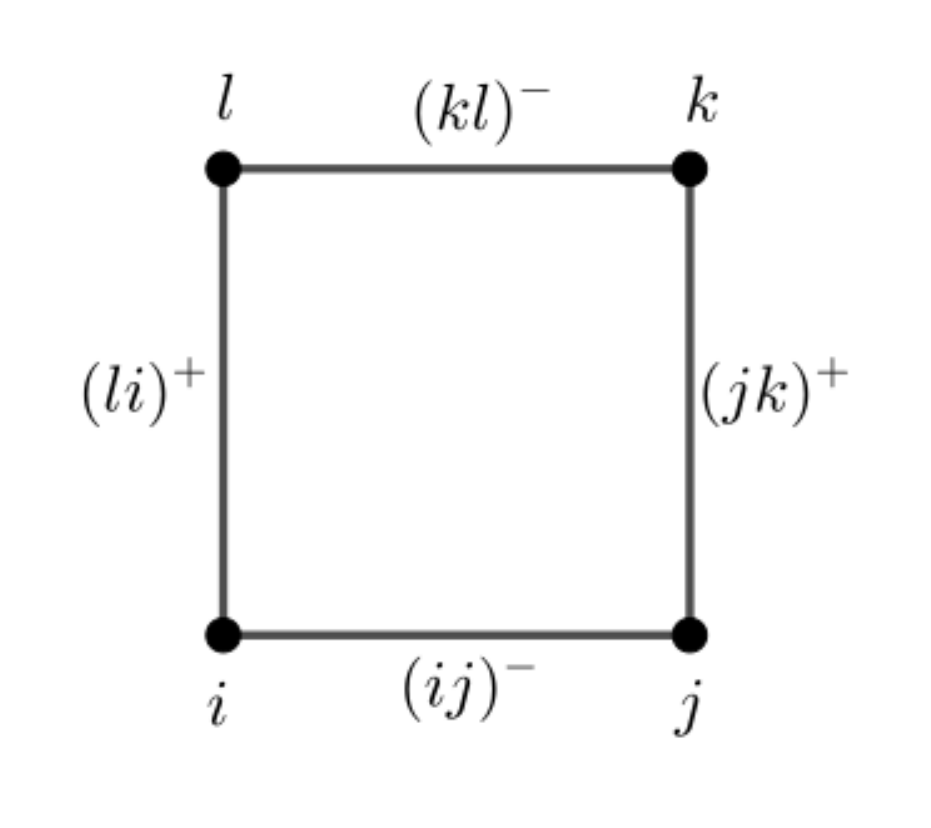}
\end{minipage}
\\[10pt] \ \\\begin{minipage}{1.2cm}
\ \ 
\end{minipage}
\begin{minipage}{10cm}
\textbf{Fig 2.} \textit{Left:} A (part of a) '$+$'-coordinate ribbon bounded by two '$+$'-coordinate lines. \textit{Right:} The notation used for a face $(ijkl)$.
\end{minipage}
\begin{minipage}{0.7cm}
\ \ 
\end{minipage}
%
%
\\\\[6pt]In the next paragraphs we recall the notion of discrete Legendre maps, which are particular discrete line congruences ``living'' on 0-cells (vertices) of a cell complex and which represent discrete surfaces in Lie sphere geometry. As a novel point of view, we discuss how discrete Legendre maps can be described by alternative data prescribed on  1- or 2-cells of the underlying cell complex.
\begin{definition}[\cite{bobenko_book, lin_weingarten_discrete1}]
A \emph{discrete Legendre map} is a line congruence $f: \mathcal{V} \rightarrow \mathcal{Z}, i \mapsto f_i$ so that two adjacent contact elements $f_i$ and $f_j$ share a common \emph{curvature sphere} $s_{ij}:= f_i \cap f_j$.
\end{definition}
\noindent Note that, for any discrete Legendre map, we therefore obtain two \textit{curvature sphere congruences} $s^+:\mathcal{E}^+ \rightarrow \mathcal{L}$ and $s^-:\mathcal{E}^- \rightarrow\mathcal{L}$.

Conversely, by prescribing suitable curvature sphere congruences on the edges of a quadrilateral cell complex, we can recover the discrete Legendre map:
\begin{proposition}\label{prop_leg_from_spheres}
A map $s:\mathcal{E} \rightarrow \mathcal{L}$ induces a discrete Legendre map with curvature sphere congruences $s\vert_{\mathcal{E}^+}$ and $s\vert_{\mathcal{E}^-}$ if and only if the spheres corresponding to a vertex-star span a contact element.
\end{proposition}
%
%
\noindent To carry the concept of prescribed data for a discrete Legendre map over to faces, we take up the idea of gluing patches of smooth surfaces into discrete surfaces (cf.\,\cite{bo-2011-cas,paper_cyclidic} and also \cite{rib_coord} for the semi-discrete case). In particular, it turns out that patches of Dupin cyclides fit well with the concept of discrete principal curvature line nets in Lie sphere geometry. 

\begin{definition}[\cite{paper_cyclidic}]
Let $f:\mathcal{V} \rightarrow \mathcal{Z}$ be a discrete Legendre map. Then a Dupin cyclide which shares the four curvature spheres with a face will be called a \emph{face-cyclide} for the corresponding face.
\end{definition}
\noindent In \cite{paper_cyclidic} it was shown that for any face of $f$ there exists a 1-parameter family of face-cyclides and, if suitably chosen, these patches of Dupin cyclides extend a discrete net to a continuously differentiable surface. Nevertheless, when working with face-cyclide congruences in this paper, we will not ask for this additional smoothness. 
\\\\We recall that in the Lie geometric setup a Dupin cyclide is described by an orthogonal $(2,1)$-splitting, where the light cones in the two $(2,1)$-planes  represent the curvature sphere congruences of the Dupin cyclide, which degenerate to two 1-parameter families of spheres in this case. 

This leads to another characterization of discrete Legendre maps in terms of face-cyclides:
%
%
\begin{proposition}
A congruence of Dupin cyclides 
\begin{equation*}
\gamma = (\gamma^+, \gamma^-): \mathcal{F} \rightarrow (D^+ , D^-) \subset G_{(2,1)}(\mathbb{R}^{4,2}) \times G_{(2,1)}(\mathbb{R}^{4,2})
\end{equation*}
induces a Legendre map with $\gamma$ as face-cyclides if and only if there exist two maps $s^+: \mathcal{E}^+ \rightarrow \mathcal{L}$ and $s^-:\mathcal{E}^- \rightarrow \mathcal{L}$ such that for each face $(ijkl)$
\begin{equation*}
s^+_{jk}, \ s^+_{li} \subset D^+ \ \ \text{ and } \ \ s^-_{ij}, \ s^-_{kl} \subset D^- 
\end{equation*}
and spheres of $s^+$ and $s^-$ incident to the same vertex span a contact element. 

Such a congruence $\gamma$ of Dupin cyclides will be called a \emph{face-cyclide congruence} for the induced Legendre map. 
\end{proposition}
\begin{proof}
Suppose that $\gamma$ is such a congruence of Dupin cyclides, then the maps $s^+$ and $s^-$ satisfy the conditions of Proposition \ref{prop_leg_from_spheres} and therefore induce a discrete Legendre map.

Conversely, the curvature sphere congruences of a discrete Legendre map provide suitable maps $s^+$ and $s^-$ and the congruence $\gamma$ is given by face-cyclides.
\end{proof}
%
%
\vspace*{0.3cm}\section{Discrete channel surfaces}
\subsection{Lie geometric characterization of discrete channel surfaces}
%
%
\noindent Similar to the characterization (S4) of Blaschke \cite[\S 84]{blaschke} that uses Lie cyclides to distinguish smooth channel surfaces, we use the face-cyclides to define discrete channel surfaces:
\begin{definition}
A \emph{discrete channel surface} is a discrete Legendre map that admits a face-cyclide congruence which is constant along one family of coordinate  ribbons. 

Such a face-cyclide congruence will be called a \emph{Lie cyclide congruence} for the discrete channel surface; the label of the coordinate ribbon in which the Lie cyclide congruence is constant is then the \emph{circular direction} of the discrete channel surface.
\end{definition}
\noindent We start by investigating discrete versions of the properties (S1)-(S3) discussed in the introduction:
\begin{proposition}[Property (S3)]\label{1family_curv}
One of the curvature sphere congruences of a discrete channel surface is constant along its coordinate lines. 
\end{proposition}
\begin{proof}
First consider two faces $(i,j,j'',i'')$ and $(i',j',j,i)$ of a discrete Legendre map $f$ adjacent along the edge $(ij)^-$ and assume that they share a common face-cyclide $\gamma$ given by the splitting $D^+ \oplus^\perp D^-$ of $\mathbb{R}^{4,2}$. Then, the three curvature spheres $s^+_{i'i}$, $s^+_{ii''}$ and $s^-_{ij}$ of $f$ are in the same contact element $f_i$. Moreover, by definition, these three spheres are also curvature spheres in a common contact element of the smooth face-cyclide $\gamma$.

However, since there are at most two curvature spheres in each contact element of a smooth Legendre map and $D^+ \cap D^- = \{ 0 \}$, the curvature spheres $s^+_{i'i}$ and $s^+_{ii''}$ must coincide. 

In the case of a discrete channel surface the Lie cyclide congruence is constant along each coordinate ribbon in the circular direction. Hence, by the above, the curvature spheres along the circular direction must be constant. 
\end{proof}
\noindent Thus, for a circular coordinate line of a discrete channel surface the constant  curvature sphere assigned to the edges is contained in every contact element along this coordinate line. In this way, we gain information on vertices from data originally given on edges: 
\begin{corollary}[Property (S1)]
A discrete channel surface envelopes a 1-parameter family of spheres.
\end{corollary}
\noindent Furthermore, we can also deduce information about the curvature sphere congruence along the non-circular direction of a discrete channel surface. 

Without loss of generality, we assume that $'+'$ is the circular direction. Since along each $'+'$-coordinate ribbon the curvature spheres of the discrete channel surface coincide with those of the corresponding Lie cyclide, the curvature spheres $s^-$ along each $'+'$-coordinate ribbon lie in a fixed $(2,1)$-plane in $\mathbb{R}^{4,2}$. 

Conversely, this prescribed data for both curvature sphere congruences is now sufficient to induce a discrete Legendre map which is a discrete channel surface: 
\begin{proposition}\label{channel_curv_spheres}
A discrete Legendre map induced by a map $s:\mathcal{E} \rightarrow \mathcal{L}$ is a discrete channel surface with circular direction '$+$' if and only if $s|_{\mathcal{E}^+}$ is constant along any '$+$'-coordinate line and $s|_{\mathcal{E}^-}$ determines a $(2,1)-$plane along each '$+$'-coordinate ribbon.  
\end{proposition}
\begin{proof}
Utilizing the arguments above, it only remains to show that $s$ with the required properties induces a discrete channel surface. Thus, we have to construct a face-cyclide congruence which is constant along each '$+$'-coordinate ribbon. Indeed, for any '$+$'-coordinate ribbon the spheres $s|_{\mathcal{E}^-}$ belonging to this coordinate ribbon determine a Dupin cyclide, which is a face-cyclide: along the coordinate ribbon the constant sphere $s|_{\mathcal{E}^+}$ and the spheres of $s|_{\mathcal{E}^-}$ are suitable curvature spheres of the Dupin cyclide.  
\end{proof}
\noindent Since any $(2,1)$-plane in $\mathbb{R}^{4,2}$ uniquely determines a Dupin cyclide, we obtain the following corollary as an immediate consequence of Proposition \ref{channel_curv_spheres}:
\begin{corollary}\label{unique}
The Lie cyclide congruence of a discrete channel surface is unique. 
\end{corollary}
%
%
\ \\Next we aim to examine if a discrete channel surface carries one family of circular curvature lines and therefore property (S2) is satisfied also in the discrete case.
%
\\\\Contemplating the notion of a ``circle'' in our Lie geometric setup, a circle, considered as a curve of points, is a M\"obius geometric object. Hence, we fix a M\"obius geometry by choosing a point sphere complex $\mathfrak{p}$ to distinguish point spheres from the set of 2-spheres. A circle is then a Dupin cyclide with one family of point spheres as curvature spheres. The other family of curvature spheres of a circle is then given by an elliptic sphere pencil, namely, by all spheres containing the points of the circle. 

Clearly, in Euclidean space the notion of a circle is not Lie invariant: for example, a parallel transformation maps a circle to a torus of revolution.  
%
\\\\Property (S2) for smooth channel surfaces can be reformulated in the following way: for any projection to a M\"obius geometry, there exists a particular 1-parameter family of touching Dupin cyclides; they touch along one family of curvature lines and all Dupin cyclides of this family become circles in this M\"obius geometry.
\\\\We say that a smooth Dupin cyclide \emph{touches} a discrete Legendre map along a discrete  curvature line if the contact elements of the Dupin cyclide and the discrete Legendre map correspond along a coordinate line in $\mathcal{G}$.

\hspace*{1.5cm}\begin{minipage}{7cm}
\centering\hspace*{1.5cm}\includegraphics[scale=0.23]{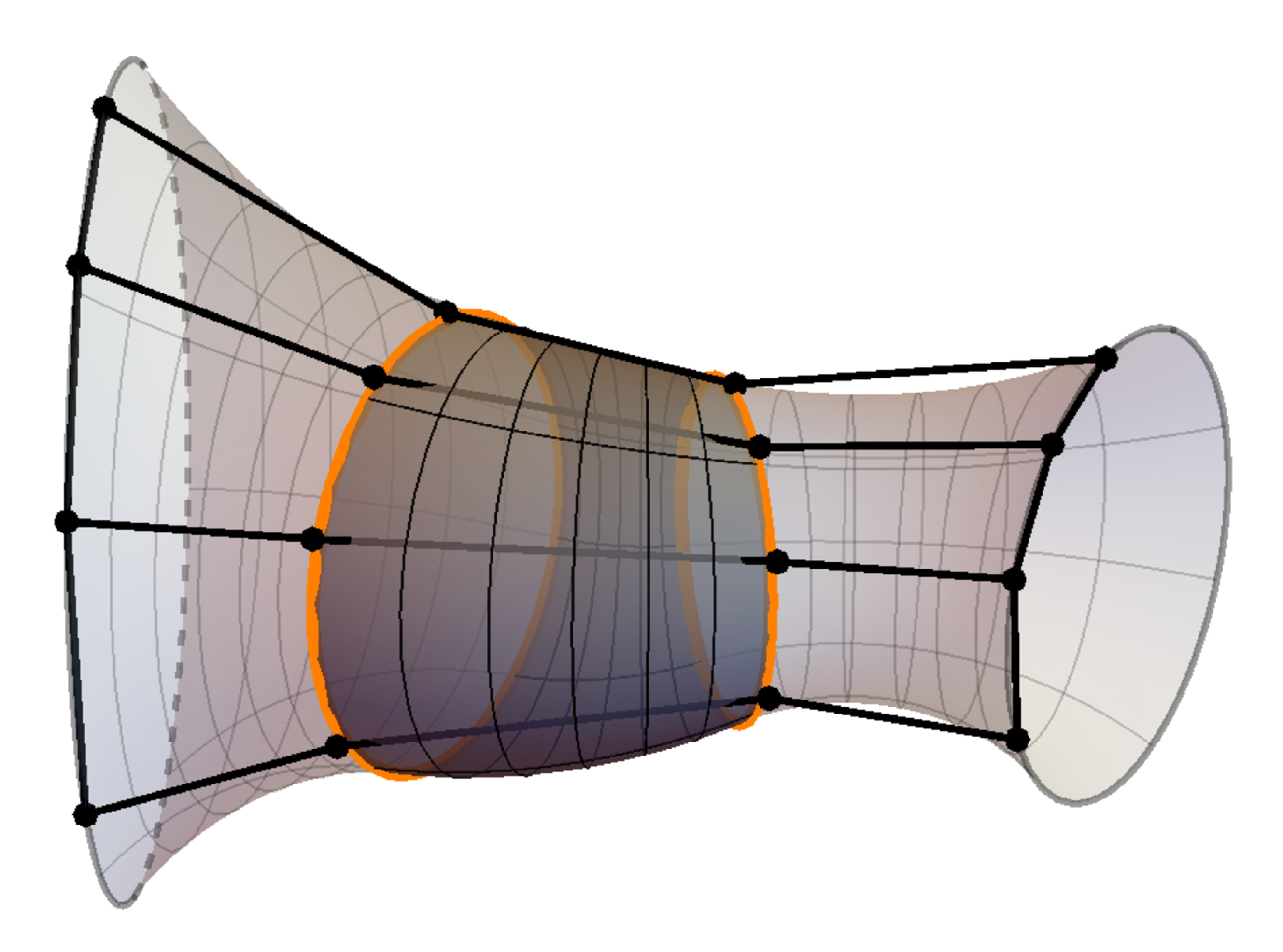}
\\ \ 
\end{minipage}
\ \\\begin{minipage}{6.5cm}
\hspace*{0.8cm}\includegraphics[scale=0.19]{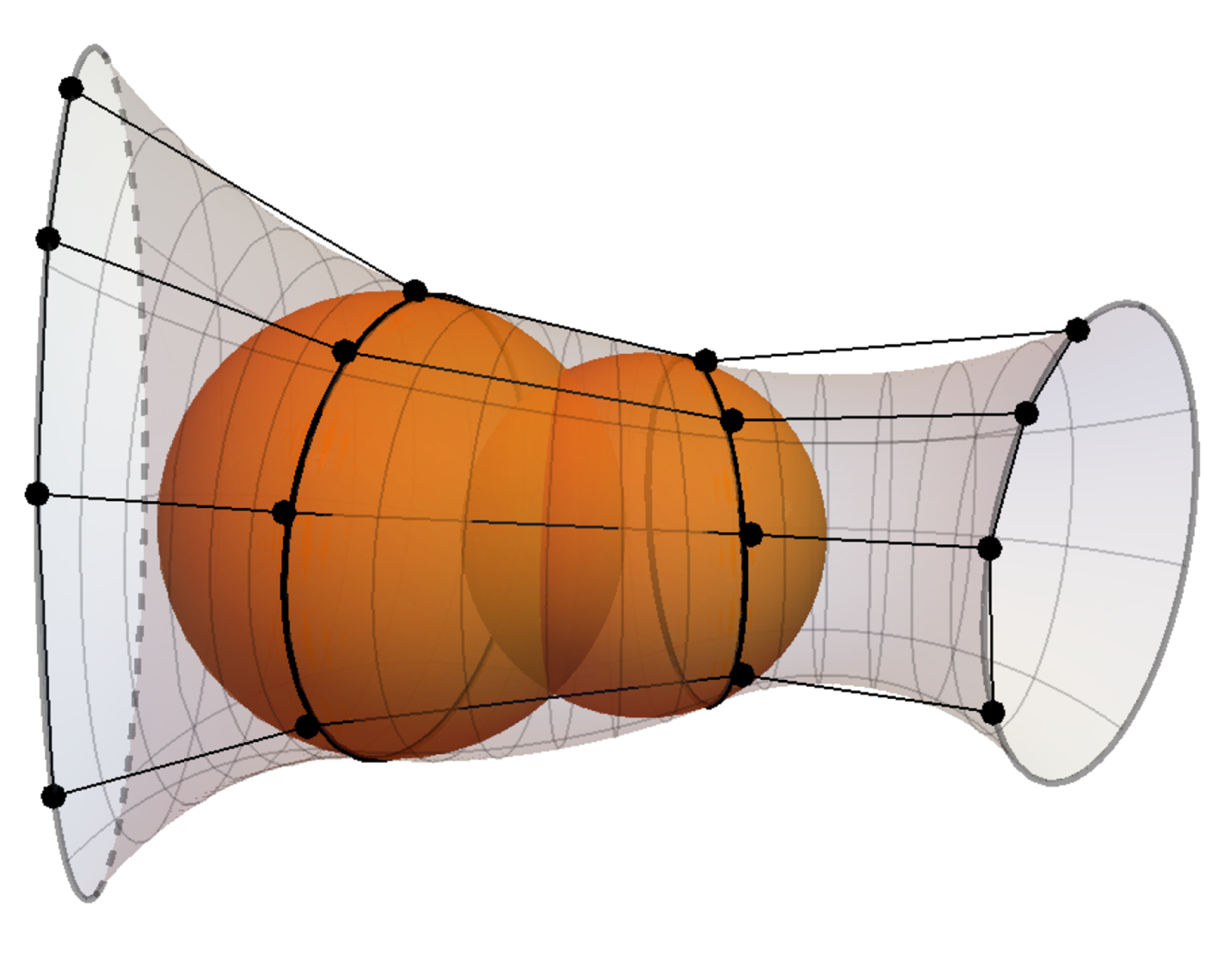}
\end{minipage}
\begin{minipage}{0.29cm}
\ \ 
\end{minipage}
\begin{minipage}{6.5cm}
\includegraphics[scale=0.69]{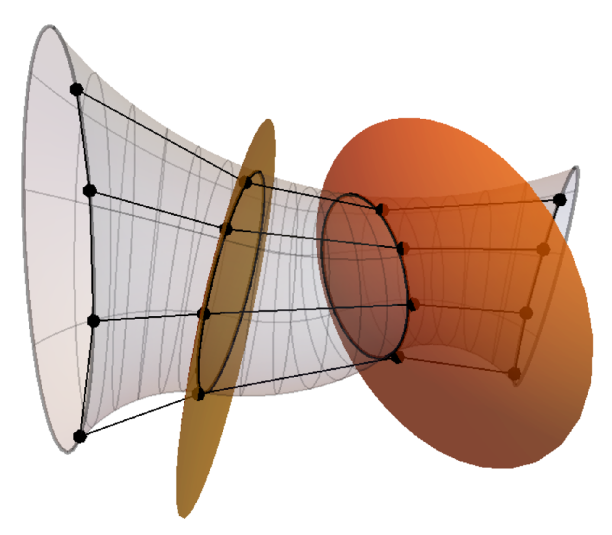}
\end{minipage}
\\\begin{minipage}{5cm}
\hspace*{0.9cm}\includegraphics[scale=0.17]{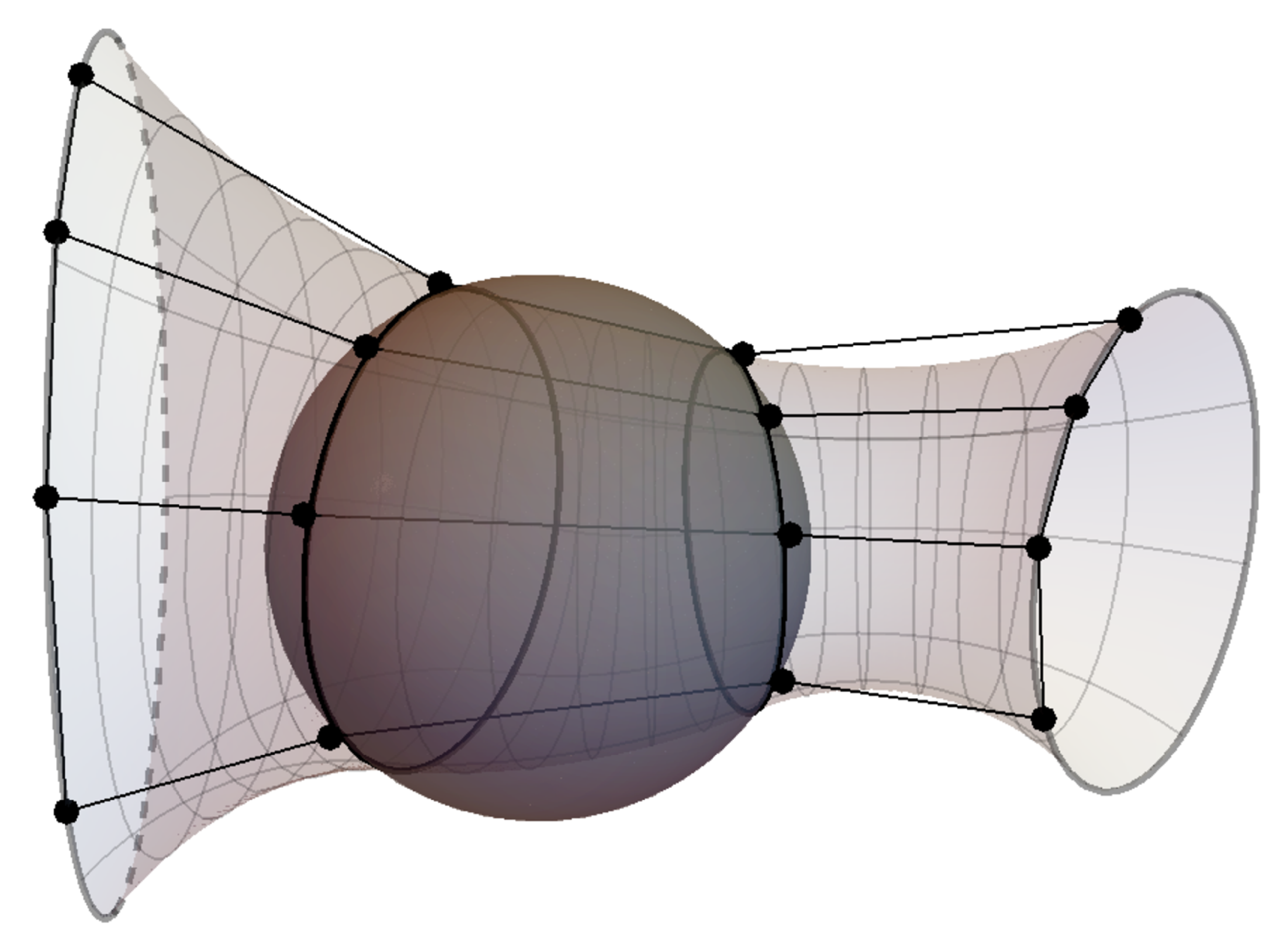}
\end{minipage}
\begin{minipage}{1cm}
\ \ 
\end{minipage}
\begin{minipage}{6.5cm}
\vspace*{0.9cm}\hspace*{0.9cm}\includegraphics[scale=0.73]{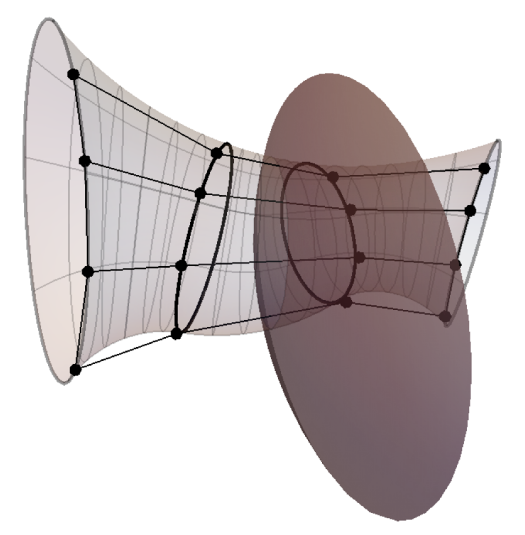}
\end{minipage}
\\[10pt] \ \\\begin{minipage}{0.5cm}
\ \ 
\end{minipage}
\begin{minipage}{11cm}
\textbf{Fig 3.} A discrete channel surface with different data given along coordinate ribbons (gray) and coordinate lines (orange). \textit{Top:} a constant Lie cyclide (gray) along a coordinate ribbon in circular direction and two generating circles (orange). \textit{Left:} two curvature spheres (orange), which are constant along the circular direction and a face-sphere (gray) containing the two generating circles. \textit{Right:} quer-spheres (orange) for each circular coordinate line and a face-quer-sphere (gray) orthogonal to the corresponding face-sphere.
\end{minipage}
\begin{minipage}{0.5cm}
\ \ 
\end{minipage}
\\[12pt]Thus, analogous to the smooth case, we will prove the existence of circular curvature lines for our discrete channel surfaces:
%
%
\begin{theoremand}[Property (S2)]\label{circ_curv_lines}
If $f$ is a discrete channel surface, then, fixing an arbitrary projection to a M\"obius geometry, there exists a 1-parameter family of Dupin cyclides that touch $f$ along one family of coordinate lines and become circles in this M\"obius geometry. 
\\These circles are called \emph{generating circles} of the discrete channel surface; in the Lie geometric context we also speak of \emph{generating cyclides}.
\end{theoremand}
\begin{proof}
Without loss of generality we assume that $f$ is a discrete channel surface where the label $'+'$ provides the circular direction. Further, let 
\begin{equation*}
(c,c^{\perp}): \{ '+'-\text{coordinate lines} \} \rightarrow G_{(2,1)}(\mathbb{R}^{4,2})\times G_{(2,1)}(\mathbb{R}^{4,2})
\end{equation*}
be an arbitrary 1-parameter family of touching Dupin cyclides along the family of '$+$'-coordinate lines, where $s^+ \in c$ and $s^- \in c^\perp$. Since a Lie cyclide touches $f$ along the two corresponding discrete circular curvature lines, the existence of such a family $(c,c^{\perp})$ is guaranteed.  

We choose a point sphere complex $\mathfrak{p}$ such that $\mathfrak{p} \not\perp s^+$ and without loss of generality we assume $(s^+,\mathfrak{p})=-1$. To construct the sought-after family of Dupin cyclides we consider the isometry 
\begin{equation*}
\varphi: c^\perp \rightarrow \mathbb{R}^{4,2}, \ \ y \mapsto y + (y,\mathfrak{p})s^+.
\end{equation*}
Since $\varphi$ injects and is timelike, the map $\tilde{c}:=(\varphi(c^\perp))^\perp$ defines a 1-parameter family of Dupin cyclides. Furthermore, since $f=\langle s^+, \varphi(s^-)\rangle$ and $\varphi(s^-) \perp \mathfrak{p}$, the Dupin cyclides touch $f$ and consist of point spheres, hence degenerate to circles in this M\"obius geometry.
\end{proof}
%
\subsection{M\"obius geometric properties of a discrete channel surface}
In order to construct discrete channel surfaces as envelopes of 1-parameter families of spheres, their curvature lines and, in particular, the generating circles will be crucial. To gain a better geometric understanding of the situation we study them in the context of transformations in a M\"obius geometry. 

Thus, in this section, we fix a M\"obius geometry $\langle \mathfrak{p} \rangle^\perp$ and study discrete channel surfaces $\mathfrak{f}:\mathcal{V} \rightarrow \langle \mathfrak{p} \rangle^\perp$ from a M\"obius geometric point of view.
\\\\Recall from \cite{rib_coord, rib_dajczer} that two smooth curves $c_1, c_2$ are said to form a \textit{Ribaucour pair} if they envelop a circle congruence, that is, at corresponding points they are tangent to a common circle. We call the induced point-to-point correspondence between two parametrized curves a \textit{Ribaucour correspondence}.
\\\\Since, later, Ribaucour pairs of two circles will be crucial, we remark on important properties of such pairs:
\begin{lemma}\label{two_rib_corr}
If $c_1$ and $c_2$ are two cospherical circles, then there exist exactly two Ribaucour correspondences between them. 
\end{lemma}
\noindent Note that if $c_1$ and $c_2$ are touching circles, that is, two parallel lines after stereographic projection, one of the two Ribaucour correspondences degenerates: the smooth enveloped circle congruence is then given by one of the two circles and all points on this circle are in Ribaucour correspondence with the contact point of the two circles. 
\begin{lemma} \label{rib_circular} 
A Ribaucour pair of two circles satisfies the following conditions:
\begin{itemize}
\item[\emph{(i)}] The two circles lie on a common sphere.
\item[\emph{(ii)}] Any two pairs of corresponding points are concircular.
\end{itemize}
\end{lemma}
\noindent The proofs of the lemmas follow by elementary geometric arguments after establishing that there exists a suitable stereographic projection of the two circles to the plane (cf.~Figure 4).
%
\\\\For discrete curves, the enveloping condition of a Ribaucour transformation translates into circularity of two neighbouring pairs of  vertices \cite{bobenko_book}: 
\begin{definition} 
Two discrete curves $\mathfrak{c}_1, \mathfrak{c}_2$ form a \emph{Ribaucour pair} if any two adjacent pairs of corresponding points are concircular. 

A Ribaucour transformation of discrete circular curves is said to be \emph{induced by a smooth Ribaucour transformation} if the discrete point-to-point correspondence can be extended to a smooth Ribaucour correspondence between the two underlying smooth circles.
\end{definition}
\noindent Although, on a smooth Ribaucour pair of circles, there are many choices for vertices to obtain a discrete Ribaucour pair, a discrete Ribaucour correspondence induced by a smooth Ribaucour transformation is more restrictive: by Lemma \ref{two_rib_corr}, for two cospherical circles there exist, up to subdivision, only two such discrete Ribaucour correspondences (cf.~Figure 4).
%
\\\\\includegraphics[scale=0.15]{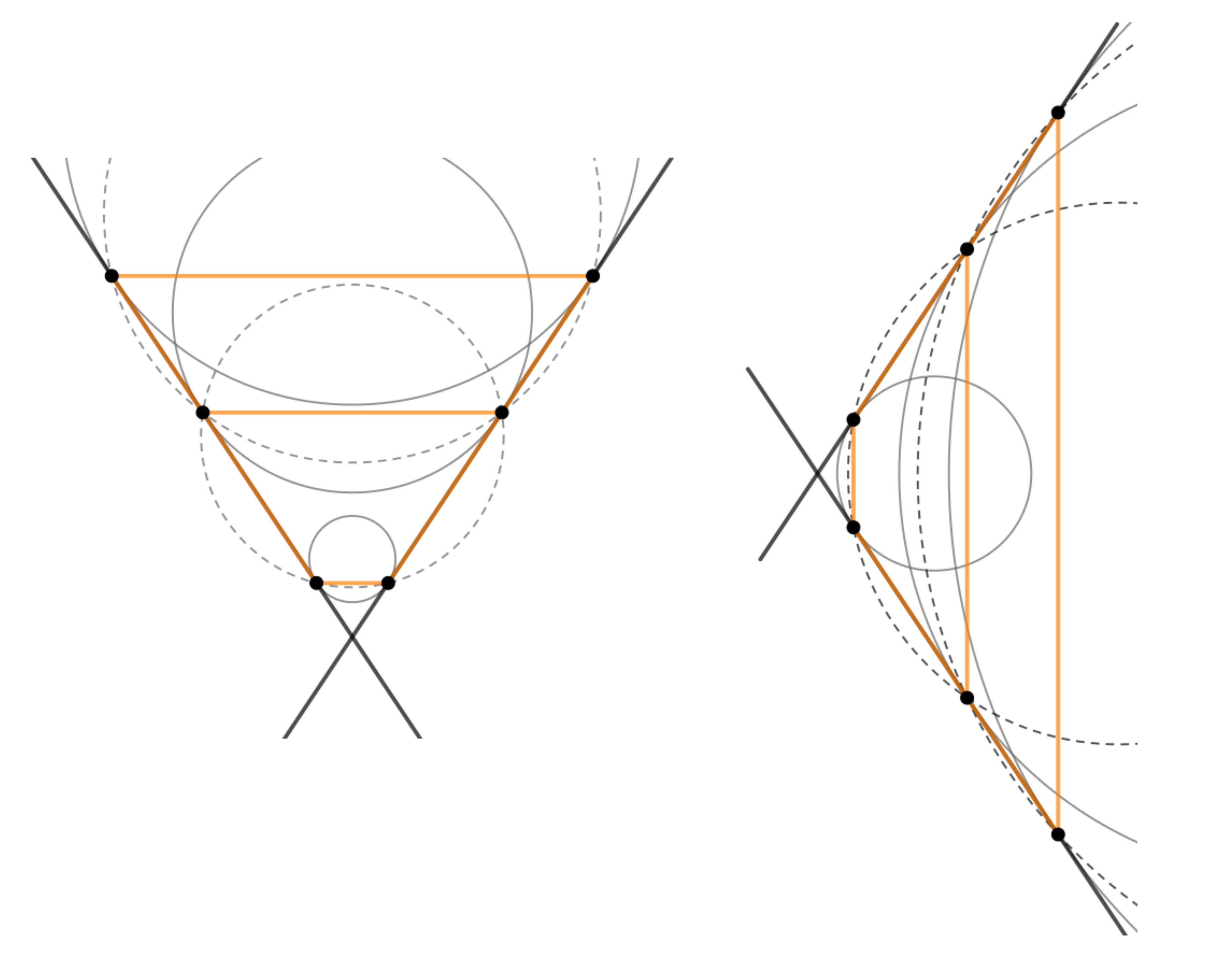}
\includegraphics[scale=0.2]{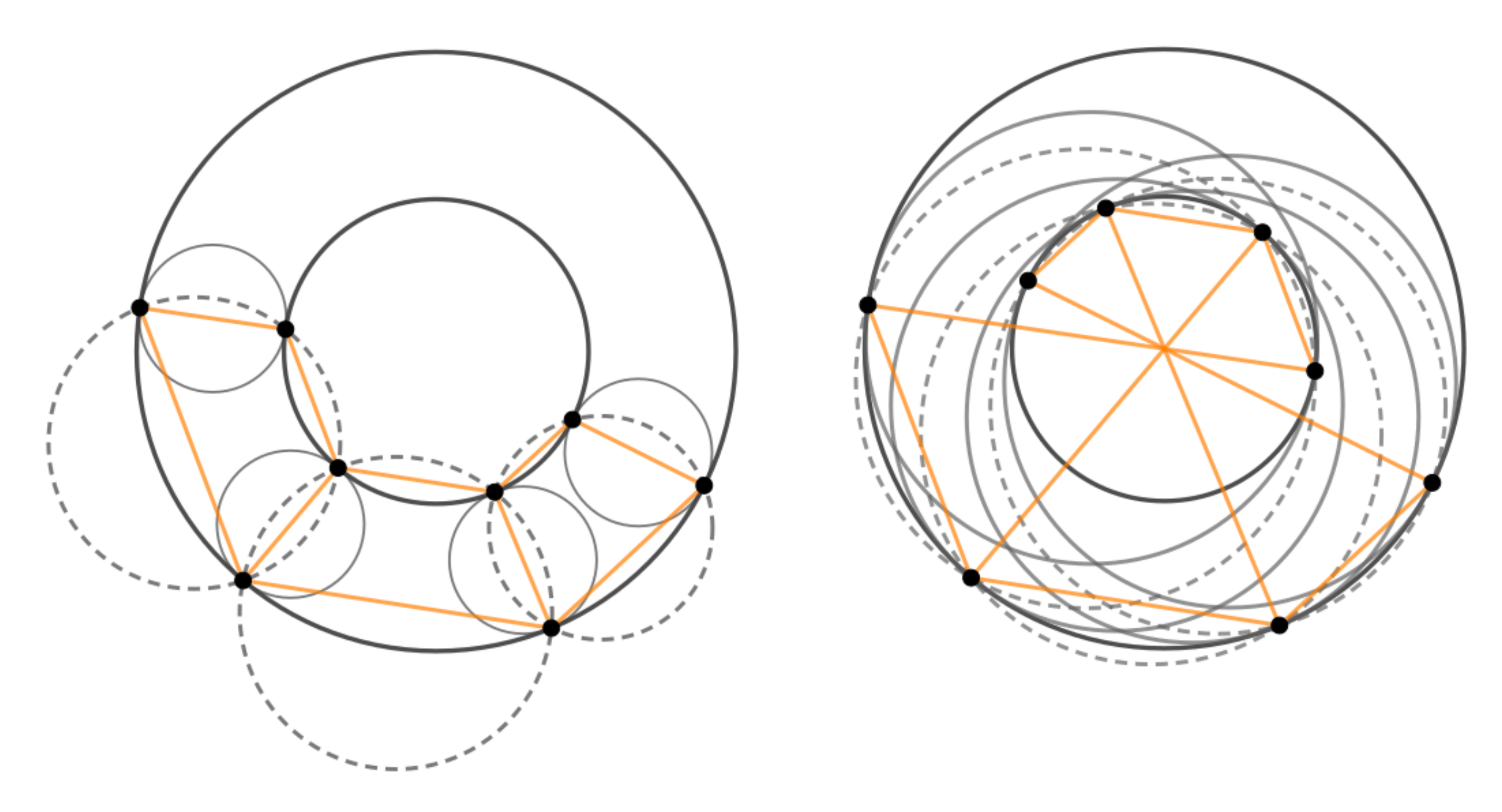}
\hspace*{0.4cm}\includegraphics[scale=0.25]{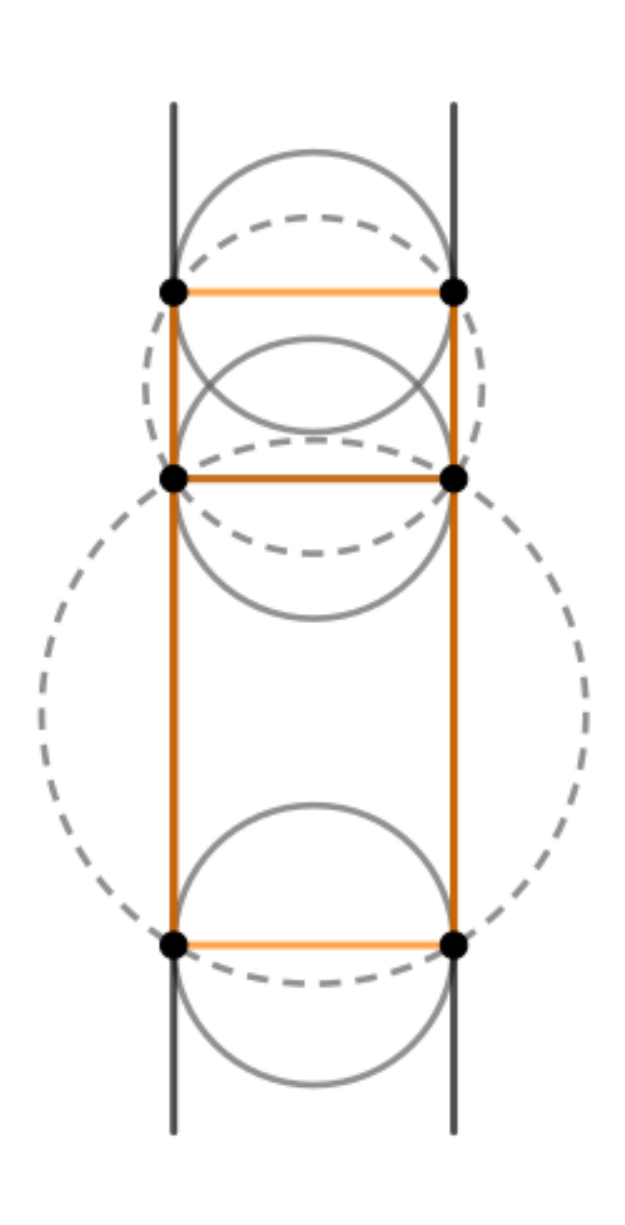}
\\\begin{minipage}{0.5cm}
\ \ 
\end{minipage}
\begin{minipage}{11cm}
\textbf{Fig 4.} Discrete Ribaucour correspondences induced by a smooth Ribaucour transformation between two circles
\end{minipage}
\begin{minipage}{0.5cm}
\ \ 
\end{minipage}
%
%
\\\\[6pt]In this realm of Ribaucour transformations, we next investigate the geometry of two curvature lines of a discrete channel surface.

Recall that a smooth Dupin cyclide is an inversion of a circular cone, a cylinder or a torus of revolution and therefore any two curvature lines of the same family of a smooth Dupin cyclide form a smooth Ribaucour pair. Hence, because by Theorem \ref{circ_curv_lines} the generating circles of a discrete channel surface are curvature lines on the corresponding Lie cyclides, two adjacent generating circles are related by a Ribaucour transformation.
\\\\Moreover, by construction, two corresponding vertices on two adjacent generating circles also lie on a curvature line of the Lie cyclide, which proves together with Lemma \ref{rib_circular} (i) the following proposition: 
\begin{propositionand}\label{channel_rib_induced}
Any two adjacent circular curvature lines of a discrete channel surface are related by a discrete Ribaucour correspondence that is induced by a smooth Ribaucour transformation. 

In particular, two adjacent circular curvature lines lie on a common sphere, called a \emph{face-sphere} of the discrete channel surface.
\end{propositionand}
\noindent We remark that Lie geometrically a face-sphere is the common curvature sphere of its two corresponding generating circles, interpreted as Dupin cyclides (cf.~Figure 3).
%
%
\\\\Moreover, by Proposition \ref{channel_rib_induced} and Lemma \ref{rib_circular} (ii), we deduce the following relation between non-circular curvature lines of a discrete channel surface:
\begin{corollary}\label{cor_non_circ_rib}
Any two discrete curvature lines in the non-circular direction of a discrete channel surface form a Ribaucour pair.
\end{corollary}
%
%
\noindent This reveals a relation to a subclass of circular nets, the so-called multi-circular nets, recently introduced in \cite{multinet}; that is, circular nets with the additional property that any coordinate quadrilateral is also circular. For a discrete channel surface, we obtain this additional circularity along any coordinate ribbon:  
\begin{corollary}
A discrete channel surface is composed of multi-circular coordinate ribbons which have circular curvature lines along their coordinate lines.
\end{corollary}
%
%
\noindent Apart from the 1-parameter family of enveloped curvature spheres and the face-sphere congruence, we observe another family of spheres which is well-known from the smooth theory of channel surfaces.
As discussed by Blaschke \cite[\S 57]{blaschke}, for a smooth channel surface there exists the 1-parameter family of quer-spheres. 

If the enveloped sphere congruence is parametrized by arc-length, these spheres are obtained as first derivatives of the 1-parameter family of enveloped spheres that the  quer-spheres intersect orthogonally.
\\\\For discrete channel surfaces in a M\"obius geometry, we observe the existence of such spheres on vertices and on faces. Following the naming given by Blaschke, we will call this data quer-spheres and begin to discuss the quer-spheres living on vertices:
\begin{definition}
A sphere that orthogonally intersects the curvature sphere, which is constant along the circular direction, at the corresponding generating circle will be called the \emph{quer-sphere} of this generating circle.
\end{definition}
\noindent These quer-spheres of a discrete channel surface are closely related to the Lie cyclide congruence: since two adjacent Lie cyclides share a curvature sphere along the common circular curvature line, they have - considered as smooth channel surfaces -- a common classically defined quer-sphere. By construction, this quer-sphere coincides with the quer-sphere of the discrete channel surface.
\begin{propositionand}
For any two adjacent generating circles of a discrete channel surface there exists a sphere, which is orthogonal to the face-sphere and for which inversion interchanges the two generating circles.  

We call these spheres \emph{face-quer-spheres}.
\end{propositionand}
\begin{proof}
Note that, due to the existence of the Lie cyclide congruence, any two adjacent generating circles are two members of a family of curvature lines of a Dupin cyclide. Thus, they are cospherical and there exists an inversion mapping the curvature lines onto each other (see for example \cite{multinet}). Moreover, this sphere is orthogonal to the sphere containing the two curvature lines, which is the face-sphere.
\end{proof}
\subsection{A discrete channel surface as the envelope of a sphere curve}
A smooth channel surface is given as the envelope of a 1-parameter family of (curvature) spheres, hence, can be uniquely reconstructed from data given along a curve. 

Here we aim to obtain a discrete channel surface from prescribed 1-dimensional data. Recall from Example 2 that a 1-parameter family of prescribed curvature spheres does not necessarily yield a discrete channel surface. Indeed, as Proposition \ref{channel_curv_spheres} shows, information about both curvature sphere congruences on a 2-dimensional cell complex is required to describe a discrete channel surface.
\\\\However, considering the projection of a discrete channel surface to a M\"obius geometry, the face-spheres, described by a discrete curve of pairwise intersecting spheres, are available. 

Conversely, a discrete curve of prescribed pairwise intersecting spheres determines a family of circles, which allows the reconstruction of a discrete channel surface with these circles as generating circles. But, due to the ambiguity in the choice of the curvature spheres (which corresponds to the choice of a normal), the reconstructed discrete channel surface is far from being unique and this approach is therefore unsatisfactory.
%
%
%
\\\\To solve these issues, we prescribe spheres on a 1-dimensional cell complex $\mathcal{G}^1=(\mathcal{V},\mathcal{E})$, and use a combination of the 1-parameter family of enveloped curvature spheres and of face-spheres. 

In this way we obtain a unique discrete channel surface from a projection in a M\"obius geometry $\langle \mathfrak{p} \rangle ^\perp$:   
\begin{theorem}\label{sphere_curve}
Let $\mathfrak{s}=(s,\sigma):(\mathcal{V},\mathcal{E}) \rightarrow S^{3,1}$ be a regular discrete sphere curve, that is,  
\begin{itemize}
\item[\emph{(i)}] three consecutive spheres $\sigma_{ij}, s_j ,\sigma_{jk}$ form an elliptic sphere pencil:
\begin{equation*}
s_j \in \langle \sigma_{ij}, \sigma_{jk} \rangle =: c_j^\perp \subset \mathbb{R}^{4,1}=\langle \mathfrak{p}\rangle ^\perp \ \ \text{and}
\end{equation*}
\item[\emph{(ii)}] two consecutive spheres $s_i, s_j$ intersect $\sigma_{ij}$ at the same unoriented angle:
\begin{equation*}
(s_i, \sigma_{ij})^2= (s_j, \sigma_{ij})^2.
\end{equation*}
\end{itemize}
Then there exists a discrete channel surface with $s$ as one family of curvature spheres and $\sigma$ as face-spheres.

This channel surface is unique up to subdivision in the circular direction.
\end{theorem}
\begin{proof}
Suppose $\mathfrak{s}$ is a regular sphere curve. We will construct a projection of a discrete Legendre map in the M\"obius geometry $\langle \mathfrak{p} \rangle ^\perp$ which is a discrete channel surface with $\sigma$ as face-spheres and $s$ as 1-parameter family of enveloped spheres.
\\\\Taking into account condition (i), the map $\sigma:\mathcal{E} \rightarrow S^{3,1}$ consists of pairwise intersecting spheres and therefore determines a sequence of smooth circles $(c_i)_{i \in \mathcal{E}}$, which shall become the generating circles for the sought-after discrete channel surface. Since two adjacent circles $c_i$ and $c_j$ lie on the common sphere $\sigma_{ij}$, there exist two smooth Ribaucour correspondences between them (cf.~Lemma \ref{two_rib_corr}). We denote these point-to-point correspondences by $r_1:c_i \rightarrow c_j$ and $r_2:c_i \rightarrow c_j$.

Following the ideas in \cite{rib_coord}, we next aim to construct for any two adjacent generating circles $c_i$ and $c_j$ a Dupin cyclide, which will then become the Lie cyclide of the discrete channel surface.

First note that the normal field of the oriented sphere $s_i$ induces a parallel normal field on the circle $c_i$. By \cite[Lemma 2.\,1]{rib_coord}, any Ribaucour correspondence maps $n_i$ to a parallel normal field $n_j$ of the circle $c_j$. In doing so, the two Ribaucour correspondences $r_1$ and $r_2$ determine the orientation of $n_j$. Since the spheres $s_i$ and $s_j$ intersect the face-sphere $\sigma_{ij}$ at the same angle, the normal field $n_j$ induced by either $r_1$ or $r_2$ coincides with the normal field of the sphere $s_j$. We fix the normal field $n_j$ such that its orientation aligns with the normal field of $s_j$.

The intersection points of the constructed parallel normal fields $n_i$ and $n_j$ then give rise to a 1-parameter family of oriented spheres, which determines a Dupin cyclide with the circles $c_i$ and $c_j$ as curvature lines and the spheres $s_i$ and $s_j$ as curvature spheres. Thus, in this way we obtain a unique 1-parameter family of Dupin cyclides 
\begin{equation*}
\gamma_{ij}:\mathcal{V}\rightarrow G_{(2,1)}(\mathbb{R}^{4,2}) \times G_{(2,1)}(\mathbb{R}^{4,2})
\end{equation*}
such that two adjacent Dupin cyclides $\gamma_{ij}$ and $\gamma_{jk}$ have the sphere $s_j$ as common curvature sphere and intersect in the circle $c_j$.

\bigskip

\begin{minipage}{5cm}
\hspace*{1.8cm}\includegraphics[scale=0.3]{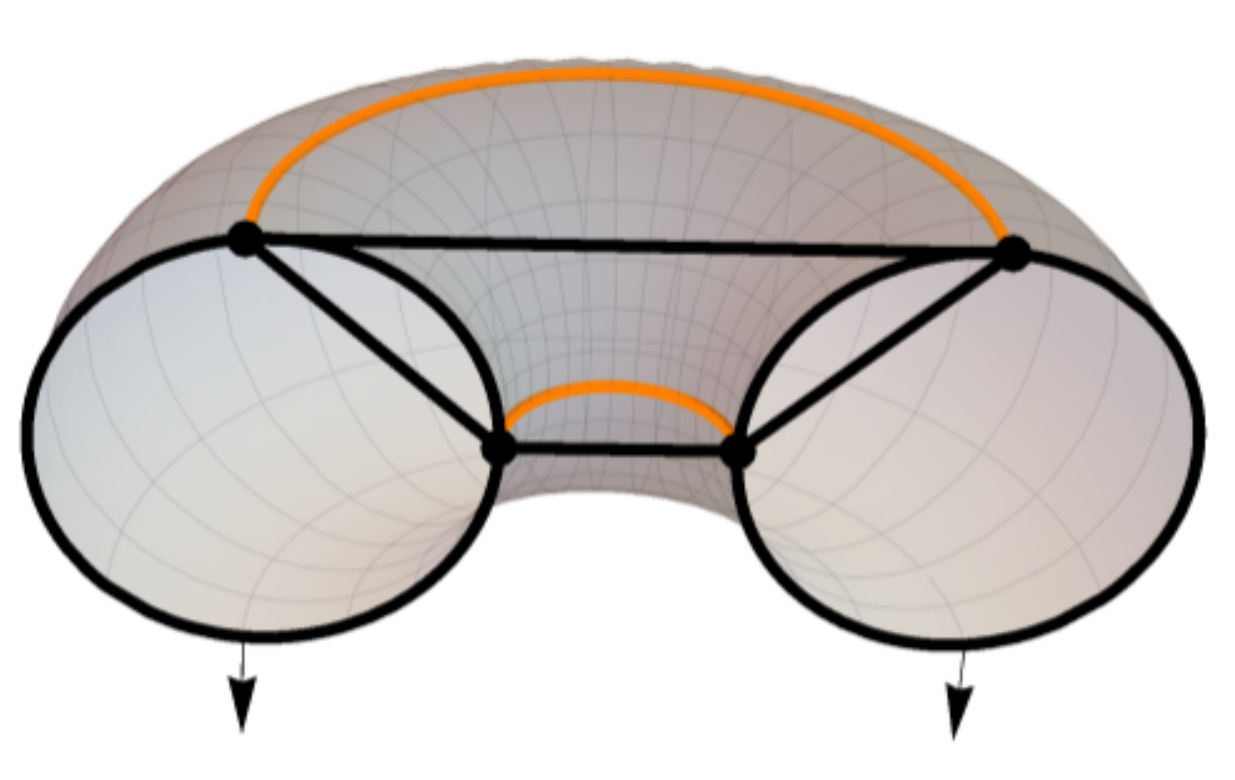}
\end{minipage}
\begin{minipage}{5cm}
\hspace*{1.3cm}\includegraphics[scale=0.33]{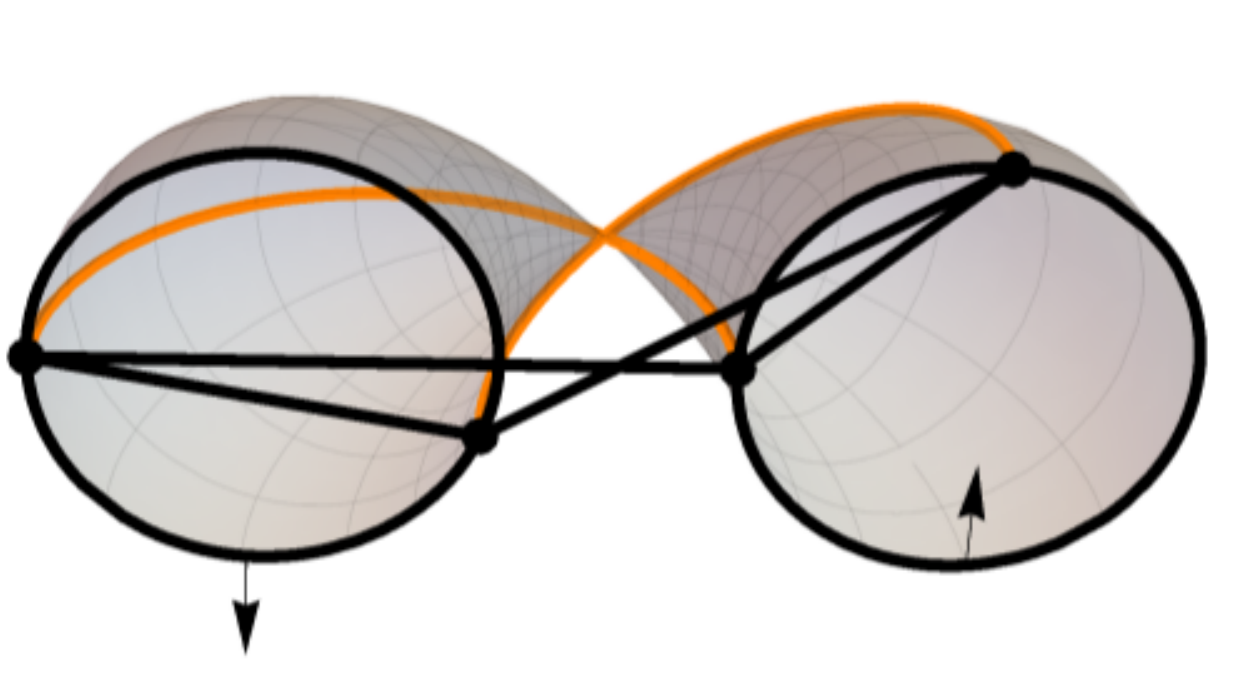}
\end{minipage}
\\[10pt] \ \\\begin{minipage}{0.5cm}
\ \ 
\end{minipage}
\begin{minipage}{11cm}
\textbf{Fig 5.} The constructed Lie cyclide between two adjacent generating circles depends on the orientation of the two corresponding curvature spheres.
\end{minipage}
\begin{minipage}{0.5cm}
\ \ 
\end{minipage}
\\\\[6pt]In order to obtain a discrete channel surface, we furthermore choose vertices on the generating circles such that corresponding vertices on two adjacent circles are related by a discrete Ribaucour transformation induced by the smooth Ribaucour correspondence $r_1$ or $r_2$. This choice of vertices is unique up to subdivision along the circles $c_i$.

In summary, we get a discrete channel surface with generating circles $c_i$, face-spheres $\sigma_{ij}$, Lie cyclides $\gamma_{ij}$ and it envelopes the 1-parameter family of curvature spheres $s_i$.
\end{proof}
\noindent We remark that, by Proposition \ref{rib_circular}, a discrete channel surface can always be closed in the circular direction. 

Also, the blending Dupin cyclide between two generating cyclides can be constructed Lie geometrically using a projection, similar to the construction of generating circles in the proof of Theorem \ref{circ_curv_lines}: 
given $c_i=(D_i \oplus D_i^{\perp})$ and $s_i \in D_i^{\perp}$, the projection
\begin{equation*}
\tau \mapsto \pi(\tau):= \tau - \frac{(\tau, s_i)}{(s_i,s_j)}s_i
\end{equation*}
yields a Dupin cyclide $(\pi(D_i), \pi(D_i)^\perp)$ touching the original Dupin cyclide as well as any prescribed sphere $s_j$. The symmetry of those construction in $i$ and $j$ follows from the fact that $c_i$ and $c_j$ are Ribaucour transformations of each other, as elaborated above.
%
%
%
%
\subsection{Discrete Dupin cyclides} As an application, we discuss how the subclass of discrete Dupin cyclides arises in this context. We define them in Lie sphere geometry analogous to the smooth case:
\begin{definition}
A discrete Legendre map is called a \emph{discrete Dupin cyclide} if it is a discrete channel surface with respect to both coordinate directions.
\end{definition}
\noindent As expected, a discrete Dupin cyclide has similar properties as its smooth counterpart:
\begin{proposition}\label{prop_dupin_cyclides}
A discrete Dupin cyclide $f:\mathcal{V}\rightarrow \mathcal{Z}$ has the following properties:
\begin{itemize}
\item[\emph{(i)}] any projection of $f$ in a M\"obius geometry has two families of discrete circular curvature lines.
\item[\emph{(ii)}] $f$ envelopes two 1-parameter families of spheres.
\item[\emph{(iii)}] the two curvature sphere congruences $s^{\pm}:\mathcal{E}^\pm\rightarrow \mathcal{L}$ lie in two fixed orthogonal $(2,1)$-planes in $\mathbb{R}^{4,2}$.
\item[\emph{(iv)}] the Lie cyclide congruence of $f$ is constant.
\end{itemize}
\end{proposition}
\begin{proof}
Suppose $f$ is a discrete Dupin cyclide. Then assertion (i) follows directly from Proposition \ref{circ_curv_lines}, while statement (ii) holds by Proposition \ref{1family_curv}. 

Property (iii) is a consequence of Proposition \ref{channel_curv_spheres}: since $f$ is a discrete channel surface with respect to both curvature directions, the curvature sphere congruences $s^+$ and $s^-$ are constant along the $'+'$- and $'-'$-direction, respectively. Thus, they lie in two fixed orthogonal $(2,1)$-planes, which shows condition (iii).

Hence, the two curvature sphere congruences of a discrete Dupin cyclide  coincide with the curvature sphere congruences of a (unique) smooth Dupin cyclide, which yields the constant Lie cyclide congruence of $f$. This proves assertion (iv).
\end{proof}
\noindent Thus, a discrete Dupin cyclide shares contact elements with a (unique) smooth Dupin cyclide. Hence, the contact elements along a circular discrete curvature line, thought of as lines in projective space, intersect in  a single point: they determine a multi-line congruence in the sense of \cite{multinet}. Therefore, a smooth Dupin cyclide yields discrete Dupin cyclides by (suitably) sampling along its curvature lines.
%
%
\ \\\section{A channel surface from two prescribed curvature lines}
\noindent Due to their special geometric properties smooth channel surfaces and, in particular Dupin cyclides, are useful as blending surfaces between two prescribed curves in computer graphics \cite{blending, blending_2}, as well as, in the theory of semi-discrete surfaces \cite{rib_coord}. 

Throughout this section we work in a M\"obius geometry $\langle \mathfrak{p} \rangle^\perp$, where we prescribe two discrete curves and aim to answer the following central question: is it possible to find a discrete channel surface which has these two prescribed discrete curves as curvature lines?
\\\\We start with the case of two prescribed circular discrete curvature lines. In \cite{blending}, it was proven that, for any two cospherical (unparametrized) smooth circles, there exists a blending smooth Dupin cyclide with these circles as curvature lines. 
\\\\In the discrete setup, a better understanding of the relation between the non-circular curvature lines of a discrete channel surface will shed light on the situation. As Vessiot noticed (cf.\,\cite{channel_curv, vessiot_channel}), non-circular curvature lines of a smooth channel surface satisfy a Riccati-equation. As a consequence, we obtain a particular relation between the non-circular curvature lines: suppose we fix any four non-circular curvature lines and compute the cross-ratio of the corresponding four points lying on a common generating circle. Then this cross-ratio is the same for any generating circle.
 
\noindent We observe the same property in the discrete setting:
\begin{proposition}\label{blending_circular}
The cross-ratio of four vertices on a generating circle of a discrete channel surface is constant along these four non-circular curvature lines.
\end{proposition}
\begin{proof}
By Proposition \ref{channel_rib_induced}, two corresponding vertices on two adjacent generating circles lie on a smooth curvature line of the corresponding Lie cyclide. Thus, in particular, these vertices lie on four curvature lines of a smooth channel surface and as a consequence of Vessiot's observations in the smooth case, the cross-ratio for both generating circles is the same. Thus, the cross-ratio is constant along these non-circular discrete curvature lines.
\end{proof}  
\noindent Therefore, in general, a blending discrete channel surface with any number of intermediate generating circles between two discrete circles does not exist.

If we only prescribe two smooth circles and allow for a choice of vertices, we can rely on the smooth theory \cite{blending} and obtain blending discrete surfaces by suitably sampling a sequence of intermediate smooth blending Dupin cyclides (cf.~Proposition \ref{prop_dupin_cyclides}).    
%
%
%
\\\\Next we consider the second case, where two discrete curvature lines in the non-circular direction of the sought-after discrete channel surface are prescribed.

In \cite{rib_coord}, it was recently proven that for any Ribaucour pair of smooth curves there exists a 1-parameter family of channel surfaces having these curves as curvature lines. As the next theorem shows, there is more freedom in the discrete setup:
%
%
\begin{theorem}\label{prescribed}
Let $(c_1, c_2)$ be a discrete Ribaucour pair of curves, then there exists (up to subdivision in the circular direction) a 3-parameter family of discrete channel surfaces that contain both curves $c_1$ and $c_2$ as curvature lines.
\end{theorem} 
\begin{proof}
Suppose that the discrete curves $c_1$ and $c_2$ form a Ribaucour pair. To obtain the sought-after discrete channel surfaces we construct suitable Lie cyclide congruences, which then, by Theorem \ref{sphere_curve}, determine discrete channel surfaces. These are unique up to subdivision in the circular direction.

To do so we choose a contact element for an initial vertex on one of the curves $c_1$ and $c_2$, which leaves us with a 2-parameter choice. Since the curves $c_1$ and $c_2$ are related by a discrete Ribaucour transformation, this contact element can be consistently extended on all vertices of $c_1$ and $c_2$ such that we obtain one coordinate ribbon of a discrete Legendre map. Then, for each face of this coordinate ribbon there exists a 1-parameter family of face-cyclides.

If we fix a face-cyclide on one initial face, then there exists a unique choice of face-cyclides on the other faces of the coordinate ribbon such that two consecutive face-cyclides intersect along a common curvature line. By construction, this yields a Lie cyclide congruence for a discrete channel surface which has as generating circles these common curvature lines as shown in the proof of Theorem \ref{sphere_curve}.
\end{proof}
\noindent  \ \\To demonstrate the difference between the smooth and the discrete theories we explicitly construct all channel surfaces with two parallel lines as prescribed curvature lines in the non-circular direction:
\begin{example} \emph{Suppose we prescribe two smooth parallel lines, then the 1-parameter family of smooth channel surfaces with these two lines as curvature lines is given by cylinders with varying radii. 
\\\hspace*{0.5cm}If we prescribe a Ribaucour pair of parallel discrete lines related by a smooth Ribaucour transformation, then there exists a 3-parameter family of discrete channel surfaces with the parallel lines as curvature lines in the non-circular direction.} 
\end{example}
\noindent \begin{minipage}{7.5cm}
To obtain these discrete channel surfaces we follow the construction given in the proof of Theorem \ref{prescribed} and construct suitable Lie cyclide congruences: if we choose a cylinder with the two parallel lines as curvature lines as a (constant) Lie cyclide, we obtain discrete cylindrical surfaces. However, as demonstrated in Figure 6, there exist non-cylindrical discrete channel surfaces with the two discrete lines as curvature lines: here the Lie cyclides are given by parts of tori of revolution.
\end{minipage}
\begin{minipage}{5cm}
\hspace*{0.7cm}\includegraphics[scale=0.3]{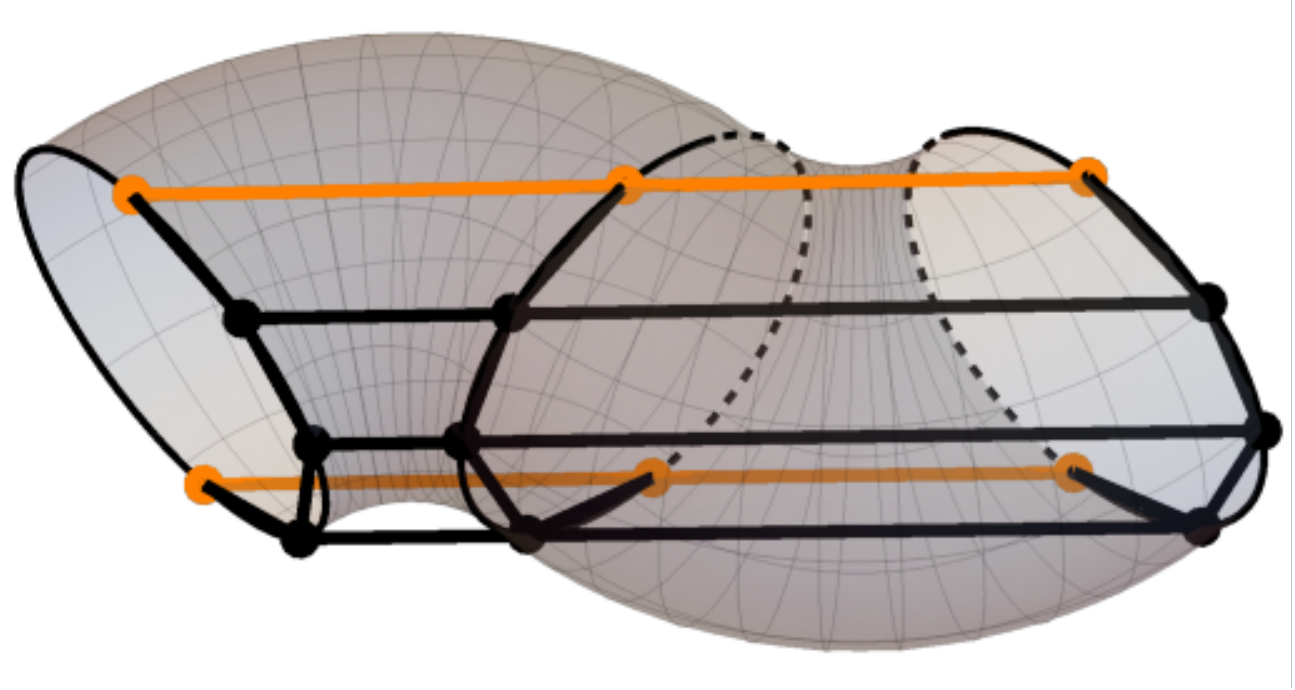}
\par
\begingroup
\leftskip=0.5cm 
\noindent \textbf{Fig 6.} A non-cylindrical discrete channel surface with two parallel lines (orange) as prescribed curvature lines.
\par
\endgroup
\end{minipage}
%
%
\ \\ \section{Discrete isothermic channel surfaces}
\noindent In this section, as an application of our notion of a discrete channel surface, we present a classification of discrete isothermic channel surfaces, where familiar subclasses of channel surfaces in Euclidean subgeometries will make an appearance.
\\\\Recall (see \cite{vessiot} and \cite[3.7]{uhj_book}) that smooth isothermic channel surfaces in M\"obius geometry are parts of surfaces of revolution, cylinders or cones in a suitably chosen Euclidean subgeometry.

To begin with we define discrete counterparts of these surfaces and then prove that these are exactly the discrete isothermic channel surfaces.

\subsection{Discrete surfaces of revolution, cylinders and cones}
We consider discrete Legendre maps in a Euclidean subgeometry and define these three types of discrete channel surfaces by specifying their Lie cyclide congruences:
\begin{definition}
A discrete channel surface is a $\left\{
\begin{tabular}{@{}l@{}}
    \emph{discrete surface of revolution} \\
    \emph{discrete cylinder} \\
    \emph{discrete cone}
\end{tabular}
\right\}$
\\[6pt]if it admits a Lie cyclide congruence of $\left\{
\begin{tabular}{@{}l@{}}
    tori of revolution having the same axis \\
    cylinders having parallel axes of revolution \\
    circular cones with the same vertex
\end{tabular}
\right\}$.
\end{definition}
\noindent \ \\ \ \\As a consequence of the definition, we obtain characterizations in terms of the related face-sphere congruences of the discrete channel surfaces:
\begin{proposition}\label{prop_subclass_face_sphere}
A discrete channel surface is a $\left\{
\begin{tabular}{@{}l@{}}
    surface of revolution \\
    cylinder \\
    cone
\end{tabular}
\right\}$ if and \\[6pt] only if its face-spheres are $\left\{
\begin{tabular}{@{}l@{}}
    orthogonally intersected by a fixed line \\
    planes which intersect a fixed plane orthogonally \\
    planes which all intersect at a fixed point
\end{tabular}
\right\}$.
\end{proposition}
\noindent \\Note that the notion of a discrete surface of revolution given here coincides with the standard definition in Euclidean geometry:
\begin{corollary}\label{geometry_surf_revolution}
A discrete channel surface is a surface of revolution if it is generated by revolving a discrete planar curve about an axis coplanar with the curve.
\end{corollary}
\begin{proof}
By definition, the Lie cyclide congruence of a discrete surface of revolution is given by tori of revolution having the same axis. Since a non-circular discrete curvature line consists of adjacent vertices that lie on a smooth curvature line of the torus and all tori have the same axis, the claim is proven.
\end{proof}
%
%
\noindent Furthermore, the curvature lines of discrete cylinders and discrete cones also share analogous properties with their smooth counterparts: 
\begin{corollary}\label{geometry_cylinder}
A discrete cylinder is generated by an arbitrary discrete planar curve, which is translated orthogonally to the plane of the profile curve.  

The vertices of any non-circular curvature line of a discrete cone lie on a sphere with the vertex of the cone as center.
\end{corollary}
\noindent The proof follows directly from Propositions \ref{prop_subclass_face_sphere} and \ref{channel_rib_induced}.
%
%
%
\subsection{Vessiot's Theorem} We will prove that the three just defined subclasses of discrete channel surfaces lead to a discrete version of Vessiot's characterization of smooth isothermic channel surfaces.

From now on we restrict to graphs $\mathcal{G}$ with vertices of degree $4$ and consider nowhere-spherical discrete surfaces, that is, discrete nets such that the nine vertices of any four adjacent faces are not cospherical.
\begin{theorem}[Vessiot's Theorem]\label{vessiot}
A nowhere-spherical discrete channel surface is isothermic if and only if it is a surface of revolution, a cylinder or a cone in a suitably chosen Euclidean subgeometry. 
\end{theorem}
\noindent We will prove this theorem by elementary geometric arguments, using the 5-point sphere condition for isothermic nets introduced in \cite{5_point_iso}: a nowhere-spherical circular net is isothermic if and only if every vertex and its four diagonal neighbours lie in a common sphere, which does not contain the set of the other four neighbouring points of the vertex.
\\\\The proof of Vessiot's Theorem is based on the following observation: 
\begin{lemma}\label{lemma_iso_circular}
A nowhere-spherical discrete channel surface is isothermic if and only if the four diagonal neighbours of each vertex are concircular.
\end{lemma}
\noindent
\textit{Proof.} Let $\mathfrak{f}$ be a nowhere-spherical isothermic discrete channel surface with circular direction $'+'$. By way of contradiction, using the notation from Figure 7, we assume that the four diagonal neighbours $\mathfrak{f}_{i'}$, $\mathfrak{f}_{i''}$, $\mathfrak{f}_{k'}$ and $\mathfrak{f}_{k''}$ of any vertex $\mathfrak{f}_j$ are not concircular. Then it follows that these four vertices uniquely determine the 5-point sphere $S_5$, which guarantees isothermicity.

Furthermore, since by Corollary \ref{cor_non_circ_rib} any two non-circular curvature lines of a discrete channel surface form a Ribaucour pair, the vertices $\mathfrak{f}_{i'}$, $\mathfrak{f}_{j'}$, $\mathfrak{f}_{j''}$ and $\mathfrak{f}_{i''}$, as well as the vertices $\mathfrak{f}_{j'}$, $\mathfrak{f}_{k'}$, $\mathfrak{f}_{k''}$ and $\mathfrak{f}_{j''}$, determine two intersecting circles.
\\\begin{minipage}{7cm}
Thus, these six vertices lie on a common sphere. But this sphere coincides with the 5-point sphere $S_5$, which contradicts the assumption that $\mathfrak{f}$ is a nowhere-spherical isothermic net. In conclusion, the four diagonal neighbours of $\mathfrak{f}_{j}$ have to be concircular.
\\\\Conversely, if for any vertex $\mathfrak{f}_j$ the four diagonal neighbours are concircular, there exists a 5-point sphere containing these five vertices. If the patch is nowhere-spherical, it does not contain all of the vertices of the vertex-star of $\mathfrak{f}_j$. Therefore the discrete channel surface is isothermic. \hfill\qed
\end{minipage}
\begin{minipage}{1cm}
\ \ 
\end{minipage}
\begin{minipage}{4cm}
\hspace*{-0.2cm}\includegraphics[scale=0.45]{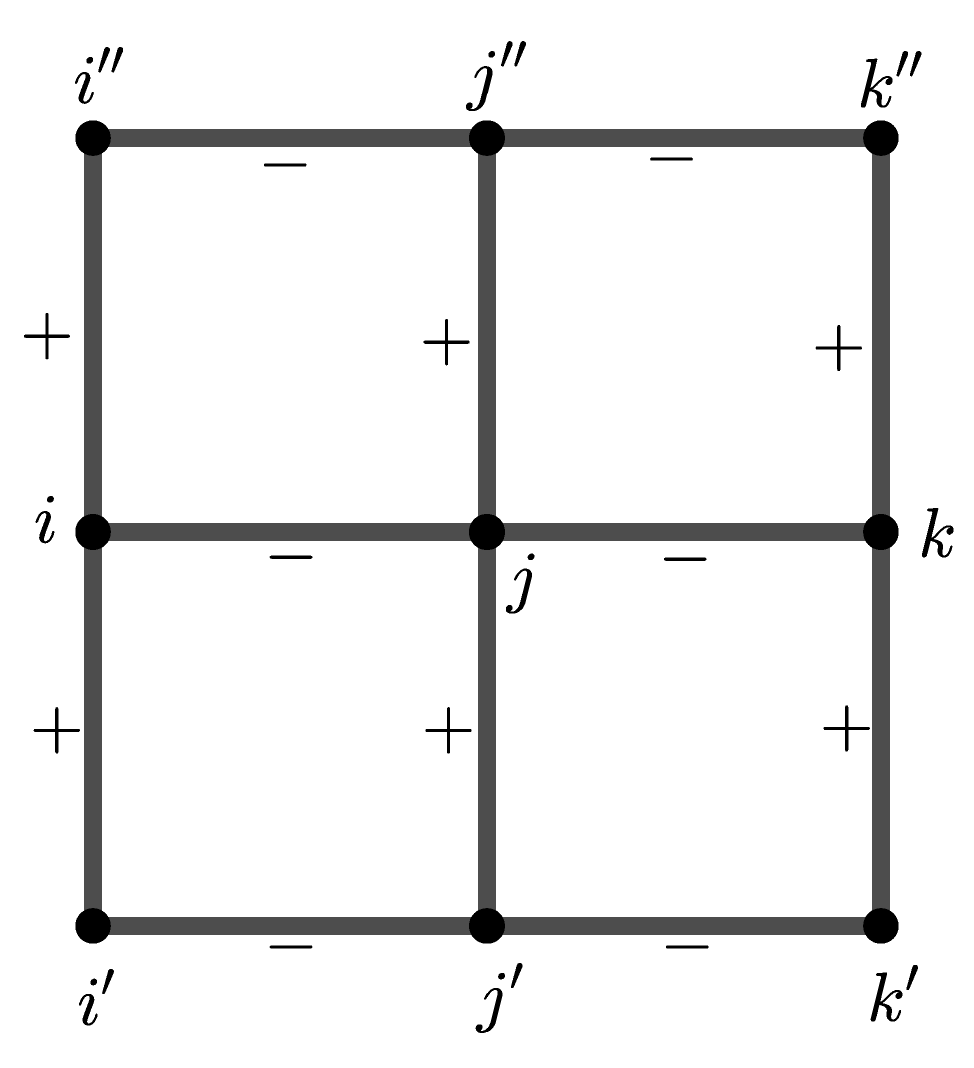}
\\\textbf{Fig 7.} Notation used for four neighbouring quadrilaterals
\end{minipage}
\\\\[6pt]
%
\\\noindent \textit{Proof of Theorem \ref{vessiot}.} Let $\mathfrak{f}$ be an isothermic discrete channel surface and denote three adjacent generating circles by $c_i, c_j$ and $c_k$. Since adjacent generating circles are related by a Ribaucour transformation, $c_i$ and $c_j$, as well as $c_j$ and $c_k$, are cospherical (see Lemma \ref{rib_circular}). Furthermore, by Lemma \ref{lemma_iso_circular}, we obtain additional circularity conditions between vertices on the circles  $c_i$ and $c_k$, which therefore also lie on a common sphere. We study the three possible configurations of $c_i$ and $c_j$:
\\[6pt](1) \ $c_i$ and $c_j$ are tangent. By a further M\"obius tranformation, the two generating circles become two parallel lines and therefore the corresponding face-sphere becomes a plane. Since $c_k$ is cospherical with each of the circles $c_i$ and $c_j$, it lies in a plane $E_i$ containing $c_i$ and in a plane $E_j$ containing $c_j$. Therefore $c_k=E_i \cap E_j$ is a line parallel to $c_i$ and $c_j$ and the two face-spheres are planes orthogonal to a fixed plane. Analogous arguments for consecutive generating circles show that in this case $\mathfrak{f}$ is M\"obius equivalent to a generalized cylinder.
\\[6pt](2) \ $c_i$ and $c_j$ intersect in two points. After a suitable M\"obius transformation we obtain two generating lines intersecting at a point $p_\infty$. Then $c_k$ is the line obtained as the intersection of a plane containing $c_i$ with a plane containing $c_j$. Therefore, in this case, we get face-spheres which determine a generalized cone.
\\[6pt](3) \ $c_i$ and $c_j$ do not intersect. From the cases (1) and (2), we then know that $c_k$ also does not intersect $c_i$ and $c_j$, hence we can consider $c_j$ as a line coplanar with each of the two circles $c_i$ and $c_k$. In this situation consider the ``equator plane'' of the sphere containing $c_i$ and $c_k$, which is orthogonal to this sphere and the line $c_j$: in this plane there exists a circle $\bar{c}$, which has as axis the line $c_j$, therefore intersects the planes of the circles $c_i$ and $c_k$ orthogonally, and at the same time orthogonally intersects the sphere containing the circles $c_i$ and $c_k$. If we transform this circle by a M\"obius transformation to a line such that the generating circles $c_i$, $c_j$ and $c_k$ become circles, then they have this line as common axis and lie in parallel planes. Hence, we obtain (a part of) a discrete surface of revolution. 
\\\\Conversely, if we have a discrete channel surface of one of the three types, we learn from Corollaries \ref{geometry_surf_revolution} and  \ref{geometry_cylinder} that, due to symmetry, the four diagonal neighbours of any vertex are concircular. Therefore, by Lemma \ref{lemma_iso_circular}, these surfaces are isothermic.   \qed
%
\noindent \\\\Using the classification of multi-circular nets recently given in \cite{multinet}, we learn from Vessiot's Theorem how these nets arise in the realm of discrete channel surfaces: 
\begin{corollary}
A discrete channel surface is a multi-circular net if and only if it is isothermic.
\end{corollary}
%
%
\subsection{Examples}
Consider a discrete Legendre map $f$ with $\mathfrak{f}$ the projection to a spaceform determined by $\mathfrak{p}$ and $\mathfrak{q}$ and $\mathfrak{n}$ its unit normal, that is,
\begin{equation*}
f= \langle \mathfrak{f}, \mathfrak{n} \rangle \ \text{with }  \mathfrak{f} \perp \mathfrak{p}, \ (\mathfrak{f},\mathfrak{q})=-1, \ \mathfrak{n} \perp \mathfrak{q}, \ (\mathfrak{n}, \mathfrak{p})=-1.
\end{equation*}
Let $\wedge: \mathbb{R}^{4,2}\times \mathbb{R}^{4,2} \rightarrow \mathfrak{o}(4,2)$, $(x \wedge y)(v)=(x,v)y-(y,v)x$ and denote the mixed area of the quadrilaterals $a$ and $b$ with parallel corresponding edges by
\begin{equation*}
A(a,b)_{ijkl}:=\frac{1}{4}(da_{ik}\wedge db_{jl}+db_{ik}\wedge da_{jl}),
\end{equation*}
where $da_{ik}:= a_i - a_k$. Then, following \cite{lin_weingarten_discrete1}, we can define the Gauss and mean curvatures on faces of $\mathfrak{f}$ as
\begin{equation}\label{wayne_onestar}
K_{ijkl}=\frac{A(n,n)_{ijkl}}{A(\mathfrak{f},\mathfrak{f})_{ijkl}} \ \ \text{and} \ \ H_{ijkl}=-\frac{A(n,\mathfrak{f})_{ijkl}}{A(\mathfrak{f},\mathfrak{f})_{ijkl}}.
\end{equation}
\ \\We will only consider non-degenerate projections of discrete Legendre maps, that is, the mixed area $A(\mathfrak{f},\mathfrak{f})_{ijkl}$ is nowhere vanishing and with sufficiently large cell complexes.
\\\\Furthermore, we define principal curvatures $\kappa_{ij}$ on edges of $\mathfrak{f}$ using Rodrigues' equation
\begin{equation}\label{wayne_twostar}
d n_{ij}=-\kappa_{ij}d\mathfrak{f}_{ij}.
\end{equation}
\ \\We have the following corollary of Vessiot's Theorem \ref{vessiot}:
\begin{corollary}
Discrete isothermic channel surfaces in space forms that are not M\"obius equivalent to discrete surfaces of revolution in a Euclidean subgeometry are all M\"obius equivalent to flat surfaces in a Euclidean subgeometry.
\end{corollary}
\begin{proof}
By Theorem \ref{vessiot}, an isothermic discrete channel surface, which is not M\"obius equivalent to a discrete surface of revolution is M\"obius equivalent to a generalized cylinder or cone in $\mathbb{R}^3$. Let that generalized cylinder or cone be $\mathfrak{f}$ above, with $\mathfrak{q}$ then necessarily null. Then, since all generating circles are straight lines with $\mathfrak{n}$ constant along those lines, we have $A(n,n)_{ijkl}\equiv 0$, implying the result.
\end{proof}
\noindent Moreover, we can characterize further well-known classes of isothermic surfaces:
\begin{proposition}
Discrete minimal and constant mean curvature channel surfaces in $\mathbb{R}^3$ are surfaces of revolution. 
\end{proposition}
\begin{proof}
Suppose $\mathfrak{f}$ is a discrete constant mean curvature channel surface with circular direction $'+'$. By the equations $(\ref{wayne_onestar})$ and $(\ref{wayne_twostar})$, the mean curvature $H$ of a face $(ijkl)$ can be written in terms of the principal curvatures as (cf. \cite{curv_theory}) 
\begin{equation*}
(\kappa_{ij}-\kappa_{il}-\kappa_{jk}+\kappa_{kl})H=\kappa_{jk}\kappa_{li}-\kappa_{ij}\kappa_{kl}.
\end{equation*}
\begin{minipage}{7.7cm}
To prove the claim, we will show that the Lie cyclides of $\mathfrak{f}$ are tori of revolution with the same axis of revolution. Thus, we first investigate the principal curvatures on the $'-'$-edges along each $'+'$-coordinate ribbon. Since the principal curvatures along each $'+'$-coordinate line of the discrete channel surface are constant, by using the notation given in Figure 8, we obtain for two adjacent faces of the $'+'$-coordinate ribbon
\begin{align*}
(\kappa_i^-+\kappa_j^--C)H&=A-\kappa_i^-\kappa_j^-
\\(\kappa_j^-+\kappa_k^--C)H&=A-\kappa_j^-\kappa_k^-, 
\end{align*}
where $A:=\kappa_1^+\kappa_2^+$ and $C:=\kappa_1^++\kappa_2^+$.
\end{minipage}
\begin{minipage}{0.7cm}
\ \ 
\end{minipage}
\begin{minipage}{3.3cm}
\hspace*{0.3cm}\includegraphics[scale=0.4]{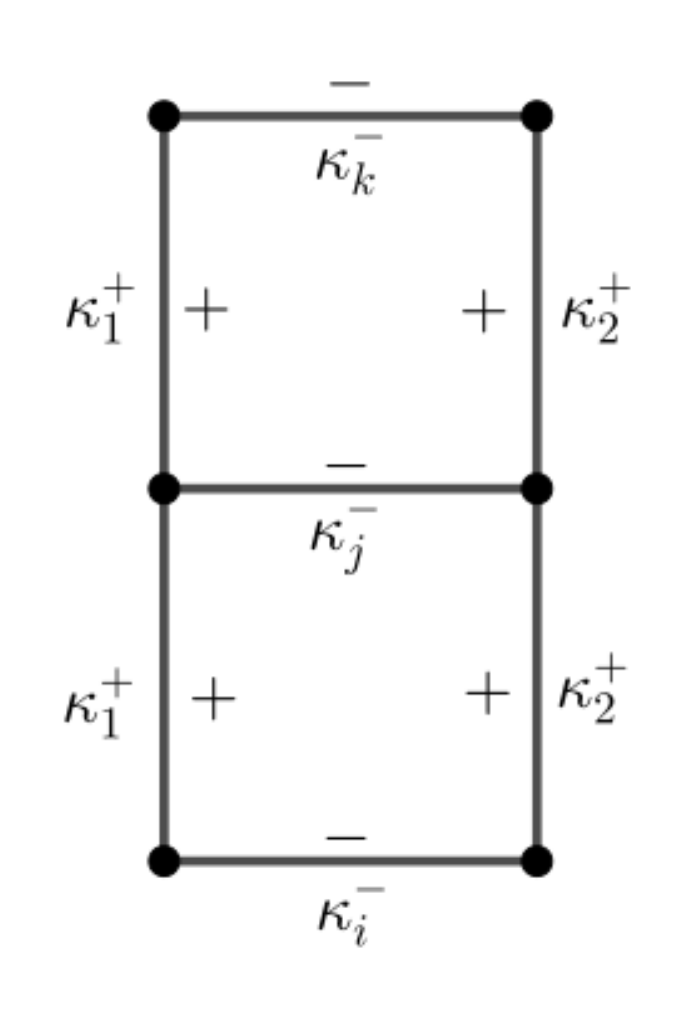}
\\\textbf{Fig 8.} Principal curvatures along a $'+'$-coordinate ribbon \end{minipage}
\\\\[6pt]Therefore, if the mean curvature is constant, we obtain $(\kappa_i^--\kappa_k^-)(\kappa_j^-+H)=0$ for any two adjacent faces of the coordinate ribbon. Thus, along a $'+'$-coordinate ribbon, for any two adjacent faces one of the conditions
\begin{equation*}
\kappa_i^-=\kappa_k^- \ \ \text{ or } \ \ \ H=-\kappa_j^-
\end{equation*}
holds and we deduce that at least three principal curvatures on the $'-'$-edges coincide.

Moreover, since $\mathfrak{f}$ is a discrete channel surface, for any $'+'$-coordinate ribbon there exists a Lie cyclide which, by construction, edgewise shares the principal curvatures with $\mathfrak{f}$. Thus, since at least three principal curvatures $\kappa_i^-$ along a $'+'$-coordinate ribbon coincide and the three curvature spheres of the Lie cyclide therefore have the same radius, the cyclide is a torus of revolution. The axis of this torus is the common axis of the two generating circles. 
\\\\In summary, in the case of a discrete constant mean curvature channel surface, the Lie cyclide congruence consists of tori of revolution having the same axis, and therefore it is, by definition, a discrete surface of revolution. 
\end{proof}
\noindent We remark that similar ideas lead to characterizations of discrete linear Weingarten channel surfaces in $\mathbb{R}^3$, as well as in space forms. A detailed analysis of these discrete surfaces will be left for a later project.
%
%
\ \\ 

\ \\ \ \\ \ \\\begin{minipage}{6.6cm}
\textbf{U.~Hertrich-Jeromin \& G.~Szewieczek}
\\TU Wien
\\Wiedner Hauptstra\ss e 8-10/104 
\\1040 Vienna, Austria
\\uhj@geometrie.tuwien.ac.at
\\gudrun@geometrie.tuwien.ac.at

\end{minipage}
\begin{minipage}{1cm}
\ \ 
\end{minipage}
\begin{minipage}{7cm}
\textbf{W.~Rossman}
\\Department of Mathematics 
\\Kobe University
\\Rokko, Kobe 657-8501, Japan
\\wayne@math.kobe-u.ac.jp
\\ \ \
\end{minipage}
\end{document}